\documentclass[11pt,a4paper,leqno,amscd,amssymb,pstricks,latexsym,amsbsy,xypic,mathrsfs,verbatim]{amsart}

\usepackage{fullpage}

\newtheorem{thm}{Theorem}[section]
\newtheorem{lem}[thm]{Lemma}
\newtheorem{cor}[thm]{Corollary}
\newtheorem{prop}[thm]{Proposition}
\newtheorem{conj}[thm]{Conjecture}

\theoremstyle{definition}
\newtheorem{example}[thm]{Example}

\newtheorem{defn}[thm]{Definition}
\newtheorem{defns}[thm]{Definitions}
\newtheorem{assum}[thm]{Assumptions}
\newtheorem{rem}[thm]{Remark}
\newtheorem{rems}[thm]{Remarks}

\numberwithin{equation}{thm}

\def\ra{\rightarrow}
\def\lra{\longrightarrow}

\input xy
\xyoption{all}

\begin{document}


\title[almost complete higher cluster tilting objects]{almost complete
cluster tilting objects \\ in generalized higher cluster categories}

\author{Lingyan GUO}
\address{Universit\'e Paris Diderot - Paris~7, UFR de Math\'ematiques,
Institut de Math\'ematiques de Jussieu, UMR 7586 du CNRS, Case 7012,
B\^atiment Chevaleret, 75205 Paris Cedex 13, France}
\email{guolingyan@math.jussieu.fr}

\date{\today}

\begin{abstract}
We study higher cluster tilting objects in generalized higher
cluster categories arising from dg algebras of higher Calabi-Yau
dimension. Taking advantage of silting mutations of Aihara-Iyama, we
obtain a class of $m$-cluster tilting objects in generalized
$m$-cluster categories. For generalized $m$-cluster categories
arising from strongly ($m+2$)-Calabi-Yau dg algebras, by using
truncations of minimal cofibrant resolutions of simple modules, we
prove that each almost complete $m$-cluster tilting $P$-object has
exactly $m+1$ complements with periodicity property. This leads us
to the conjecture that each liftable almost complete $m$-cluster
tilting object has exactly $m+1$ complements in generalized
$m$-cluster categories arising from $m$-rigid good completed
deformed preprojective dg algebras.
\end{abstract}

 \maketitle

\section{Introduction}

Cluster categories associated to acyclic quivers were introduced in
\cite{BMRRT06}, where the authors gave an additive categorification
of the finite type cluster algebras introduced by Fomin and
Zelevinsky \cite{FZ1} \cite{FZ2}. The cluster category of an acyclic
quiver $Q$ is defined as the orbit category of the derived category
of finite dimensional representations of $Q$ under the action of
${\tau}^{-1}\Sigma$, where $\tau$ is the AR-translation and $\Sigma$
the suspension functor. If we replace the autoequivalence
${\tau}^{-1}\Sigma$ with ${\tau}^{-1}{\Sigma}^m$ for some integer $m
\geq 2$, we obtain the $m$-cluster category, which was first
mentioned and proved to be triangulated in \cite{Ke05}, cf. also
\cite{Th07}. In the cluster category, the exchange relations of the
corresponding cluster algebra are modeled by exchange triangles. It
was shown in \cite{IY08} that every almost complete cluster tilting
object admits exactly two complements. In the higher cluster
category, exchange triangles are replaced by AR-angles, whose
existence (in the more general set up of Krull-Schmidt {\rm
Hom}-finite triangulated categories with Serre functors) was shown
in \cite{IY08}. Both \cite{W} and \cite{ZZ} proved that each almost
complete $m$-cluster tilting object has exactly $m+1$ complements in
an $m$-cluster category. In this paper, we study the analogous
statements for almost complete $m$-cluster tilting objects in
certain $(m+1)$-Calabi-Yau triangulated categories.

Amiot \cite{Am08} constructed generalized cluster categories using
$3$-Calabi-Yau dg algebras which satisfy some suitable assumptions.
A special class is formed by the generalized cluster categories
associated to Ginzburg algebras \cite{Gi06} coming from suitable
quivers with potentials. If the quiver is acyclic, the generalized
cluster category is triangle equivalent to the classical cluster
category. Amiot's results were extended by the author to generalized
$m$-cluster categories in \cite{GUO} by changing the Calabi-Yau
dimension from $3$ to $m+2$ for an arbitrary positive integer $m$.
As one of the applications, she particularly considered generalized
higher cluster categories associated to Ginzburg dg categories
\cite{Ke09} coming from suitable graded quivers with
superpotentials.

In the representation theory of algebras, mutation plays an
important role. Here we recall several kinds of mutation. Cluster
algebras associated to finite quivers without loops or $2$-cycles
are defined using mutation of quivers. As an extension of quiver
mutation, the mutation of quivers with potentials was introduced in
\cite{DWZ}. Moreover, the mutation of decorated representations of
quivers with potentials, which can be viewed as a generalization of
the BGP construction, was also studied in \cite{DWZ}. Tilting
modules over finite dimensional algebras are very nice objects,
although some of them can not be mutated. In the cluster category
associated to an acyclic quiver, mutation of cluster tilting objects
is always possible \cite{BMRRT06}. It is determined by exchange
triangles and corresponds to mutation of clusters in the
corresponding cluster algebra via a certain cluster character
\cite{CK06}. In the derived categories of finite dimensional
hereditary algebras, a mutation operation was given in \cite{BRT} on
silting objects, which were first studied in \cite{KV}. Silting
mutation of silting objects in triangulated categories, which is
always possible, was investigated recently by Aihara and Iyama in
\cite{AI}.

The aim of this paper is to study higher cluster tilting objects in
generalized higher cluster categories arising from dg algebras of
higher Calabi-Yau dimension. Under certain assumptions on the dg
algebras (Assumptions \ref{23}), tilting objects do not exist in the
derived categories (Remark \ref{22}). Thus, we consider silting
objects, e.g., the dg algebras themselves. The author was motivated
by the construction of tilting complexes in Section 4 of
\cite{IR06}.

This article is organized as follows: In Section 2, we list our
assumptions on dg algebras and
use the standard $t$-structure to situate the silting objects which
are iteratively obtained from $P$-indecomposables with respect to
the fundamental domain. In Section 3, using silting objects we
construct higher cluster tilting objects in generalized higher
cluster categories. We show that in such a category each liftable
almost complete $m$-cluster tilting object has at least $m+1$
complements. In Section 4, we specialize to strongly higher
Calabi-Yau dg algebras. By studying minimal cofibrant resolutions of
simple modules of good completed deformed preprojective dg algebras,
we obtain isomorphisms in generalized higher cluster categories
between images of some left mutations and images of some right
mutations of the same $P$-indecomposable. Using this, we derive the
periodicity property of the images of iterated silting mutations of
$P$-indecomposables in Section 5, where we also construct
($m+1$)-Calabi-Yau triangulated categories containing infinitely
many indecomposable $m$-cluster tilting objects. We obtain an
explicit description of the terms of Iyama-Yoshino's AR angles in
this situation, and we deduce that each almost complete $m$-cluster
tilting $P$-object in the generalized $m$-cluster category
associated to a suitable completed deformed preprojective dg algebra
has exactly $m+1$ complements in Section 6. We show that the
truncated dg subalgebra at degree zero of the dg endomorphism
algebra of a silting object in the derived category of a good
completed deformed preprojective dg algebra is also strongly
Calabi-Yau in Section 7. Then we conjecture a class (namely
$m$-rigid) of good completed deformed preprojective dg algebras such
that each liftable almost complete $m$-cluster tilting object should
have exactly $m+1$ complements in the associated generalized
$m$-cluster category.
In Section 8, we give a long exact sequence to show the relations
between extension spaces in generalized higher cluster categories
and extension spaces in derived categories. This sequence
generalizes the short exact sequence obtained by Amiot \cite{Am08}
in the $2$-Calabi-Yau case. At the end, we show that any almost
complete $m$-cluster tilting object in ${\mathcal {C}}_{\Pi}$ is
liftable if $\Pi$ is the completed deformed preprojective dg algebra
arising from an acyclic quiver.

\subsection*{Notation} For a collection
$\mathcal {X}$ of objects in an additive category $\mathcal {T}$, we
denote by add${\mathcal {X}}$ the smallest full subcategory of
$\mathcal {T}$ which contains $\mathcal {X}$ and is closed under
finite direct sums, summands and isomorphisms. Let $k$ be an
algebraically closed field of characteristic zero.


\subsection*{Acknowledgments} The author is supported by the China Scholarship Council (CSC).
This is part of her Ph.~D.~thesis under the supervision of Professor
Bernhard Keller. She is grateful to him for his guidance, patience
and kindness. She also sincerely thanks Pierre-Guy Plamondon, Fan
Qin and Dong Yang for helpful discussions and Zhonghua Zhao for
constant encouragement.

\vspace{.3cm}
\section{Silting objects in derived categories}

Let $A$ be a differential graded (for simplicity, write `dg')
$k$-algebra. We write per$A$ for the {\em perfect derived category}
of $A$, i.e.~the smallest triangulated subcategory of the derived
category ${\mathcal {D}}(A)$ containing $A$ and stable under passage
to direct summands. We denote by ${\mathcal {D}}_{fd} (A)$ the {\em
finite dimensional derived category} of $A$ whose objects are those
of ${\mathcal {D}}(A)$ with finite dimensional total homology.

A dg $k$-algebra $A$ is {\em pseudo-compact} if it is endowed with a
complete separated topology which is generated by two-sided dg
ideals of finite codimension. A (pseudo-compact) dg algebra $A$ is
{\em (topologically) homologically smooth} if $A$ lies in per$A^e$,
where $A^e$ is the (completed) tensor product of $A^{op}$ and $A$
over $k$. For example, suppose that $A$ is of the form
$(\widehat{kQ}, d)$, where $\widehat{kQ}$ is the completed path
algebra of a finite graded quiver $Q$ with respect to the two-sided
ideal $\mathfrak{m}$ of $\widehat{kQ}$ generated by the arrows of
$Q$, and the differential $d$ takes each arrow of $Q$ to an element
of $\mathfrak{m}$; it was stated in \cite{KY09} that $A$ is
pseudo-compact and topologically homologically smooth.

\begin{assum} \label{23}
Let $m$ be a positive integer. Suppose that $A$ is a
(pseudo-compact) dg $k$-algebra and has the following four
additional properties:

\begin{itemize}
\item[a)] $A$ is (topologically) homologically smooth;

\item[b)] the $p$th homology $H^p A$ vanishes for each positive
integer $p$;


\item[c)] the zeroth homology $H^0 A$ is finite dimensional;

\item[d)] $A$ is $(m+2)$-Calabi-Yau as a bimodule, {\it i.e.},~there is an isomorphism in ${\mathcal {D}}(A^e)$
$${\rm RHom}_{A^e} (A, A^e) \simeq {\Sigma}^{-m-2}A.$$
\end{itemize}
\end{assum}

\begin{thm}[\cite{Ke09}] \label{17}
(Completed) Ginzburg dg categories ${\Gamma}_{m+2}(Q,W)$ associated
to graded quivers with superpotentials $(Q,W)$ are (topologically)
homologically smooth and $(m+2)$-Calabi-Yau.
\end{thm}


\begin{lem} [\cite{Ke08}] \label{24}
Suppose that $A$ is (topologically) homologically smooth. Then the
category ${\mathcal {D}}_{fd}(A)$ is contained in {\rm per}$A$. If
moreover $A$ is ($m+2$)-Calabi-Yau for some positive integer $m$,
then for all objects $L$ of $\mathcal{D}$$(A)$ and $M$ of
${\mathcal{D}}_{fd} (A)$, we have a canonical isomorphism
\begin{center}
$D {\rm Hom}_{{\mathcal{D}}(A)} (M, L) \simeq {\rm
Hom}_{{\mathcal{D}}(A)} (L, {\Sigma}^{m+2}M).$
\end{center}
\end{lem}

Throughout this paper, we always consider the dg algebras satisfying
Assumptions \ref{23}.

\begin{prop}[\cite{GUO}]\label{4}
Under Assumptions \ref{23}, the triangulated category {\rm per}$A$
is {\rm Hom}-finite.
\end{prop}

Let $({\mathcal {D}}A)^c$ denote the full subcategory of ${\mathcal
{D}}(A)$ consisting of compact objects. Since each idempotent in
${\mathcal {D}}(A)$ is split and $({\mathcal {D}}A)^c$ is closed
under direct summands, each idempotent in $({\mathcal {D}}A)^c$ is
also split. Therefore, the category per$A$ which is equal to
$({\mathcal {D}}A)^c$ by \cite{Ke06} is a $k$-linear Hom-finite
category with split idempotents. It follows that per$A$ is a
Krull-Schmidt triangulated category.

\begin{defns} Let $A$ be a dg algebra satisfying Assumptions \ref{23}.
\begin{itemize}

\item[a)] An object $X \in {\rm per}A$ is {\em silting} (resp. {\em tilting})
if per$A=$ thick$X$ the smallest thick subcategory of per$A$
containing $X$, and the spaces ${\rm Hom}_{{\mathcal {D}}(A)} (X,
{\Sigma}^i X)$ are zero for all integers $i > 0$ (resp. $i \neq 0$).

\item[b)] An object $Y \in {\rm per}A$ is {\em almost complete silting} if there
is some indecomposable object $Y'$ in (per$A$)$\setminus$(add$Y$)
such that $Y \oplus Y'$ is a silting object. Here $Y'$ is called a
{\em complement} of $Y$.

\end{itemize}
\end{defns}

Clearly the dg algebra $A$ itself is a silting object since the
space ${\rm Hom}_{{\mathcal {D}}(A)} (A, {\Sigma}^i A)$ is
isomorphic to $H^i A$ which is zero for each positive integer.

\begin{rem} \label{22}
Under Assumptions \ref{23}, tilting objects do not exist in per$A$.
Otherwise, let $T$ be a tilting object in per$A$. By definition, the
object $T$ generates per$A$. Then for any object $M$ in
${\mathcal{D}}(A)$, it belongs to the subcategory ${\mathcal
{D}}_{fd}(A)$ if and only if ${\sum}_{p \in {\mathbb{Z}}} {\rm dim}
{\rm Hom}_{{\mathcal {D}}(A)} (T, {\Sigma}^p M)$ is finite. Since
the space ${\rm Hom}_{{\mathcal {D}}(A)}(T,T)$ is finite dimensional
by Proposition \ref{4} and the space ${\rm Hom}_{{\mathcal
{D}}(A)}(T,{\Sigma}^pT)$ vanishes for any nonzero integer $p$, the
object $T$ belongs to ${\mathcal {D}}_{fd}(A)$. Note that ${\mathcal
{D}}_{fd}(A)$ is $(m+2)$-Calabi-Yau as a triangulated category by
Lemma \ref{24}. Thus, we have the following isomorphism
$$(0=) {\rm Hom}_{{\mathcal {D}} (A)} (T, {\Sigma}^{m+2} T)
\simeq D {\rm Hom}_{{\mathcal {D}} (A)} (T,T) ( \neq 0).$$ Here we
obtain a contradiction. Therefore, tilting objects do not exist.
\end{rem}

Assume that $H^0A$ is a basic algebra. Let $e$ be a primitive
idempotent element of $H^0 A$. We denote by $P$ the indecomposable
direct summand $e A$ (in the derived category ${\mathcal {D}}(A)$)
of $A$ and call it a {\em $P$-indecomposable}. We denote by $M$ the
dg module $(1-e) A$. It follows from Proposition \ref{4} that the
subcategory add$M$ is functorially finite \cite{AS} in add$A$. Let
us write $RA_0$ for $P$ (later we will also write $LA_0$ for $P$).

By induction on $t \geq 1$, we define $RA_t$ as follows: take a
minimal right (add$M$)-approximation $f^{(t)}: A^{(t)} \ra RA_{t-1}$
of $RA_{t-1}$ in ${\mathcal {D}}(A)$ and form the triangle in
${\mathcal {D}}(A)$
\[
\xymatrix { RA_t \ar[r]^{\alpha^{(t)}} & A^{(t)} \ar[r]^-{f^{(t)}} &
RA_{t-1} \ar[r] & {\Sigma} RA_t.}
\]
Dually, for each positive integer $t$, we take a minimal left
(add$M$)-approximation $g^{(t)}: LA_{t-1} \ra B^{(t)}$ of $LA_{t-1}$
in ${\mathcal {D}}(A)$, and form the triangle in ${\mathcal {D}}(A)$
\[
\xymatrix { LA_{t-1} \ar[r]^{g^{(t)}} & B^{(t)}
\ar[r]^-{\beta^{(t)}} & LA_{t} \ar[r] & {\Sigma} LA_{t-1}.}
\]
The object $RA_t$ is called the {\em right mutation} of $RA_{t-1}$
(with respect to $M$), and $LA_t$ is called the {\em left mutation}
of $LA_{t-1}$ (with respect to $M$).

\begin{thm}[\cite{AI}]\label{5}
For each nonnegative integer $t$, the objects $M \oplus RA_t$ and $M
\oplus LA_t$ are silting objects in {\rm per}$A$. Moreover, any
basic silting object containing $M$ as a direct summand is either of
the form $M \oplus RA_t$ or of the form $M \oplus LA_t$.
\end{thm}

From the construction and the above theorem, we know that the
morphisms $\alpha^{(t)}$ (resp. $\beta^{(t)}$) are minimal left
(resp. minimal right) ${\rm (add}M \rm{)}$-approximations in
${\mathcal {D}}(A)$ and that the objects $RA_t$ and $LA_{t}$ are
indecomposable objects in ${\mathcal {D}}(A)$ which do not belong to
${\rm add}M$.

We simply denote ${\mathcal{D}} (A)$ by $\mathcal{D}$. Let
${\mathcal{D}}^{{\leq} {0}}$ (resp. ${\mathcal{D}}^{{\geq} {1}}$) be
the full subcategory of $\mathcal{D}$ whose objects are the dg
modules $X$ such that $H^p X$ vanishes for each positive (resp.
nonpositive) integer $p$. For a complex $X$ of $k$-modules, we
denote by ${\tau}_{{\leq} 0} X$ the subcomplex with $({\tau}_{{\leq}
0} X)^0 = {ker} d^0$, and $({\tau}_{{\leq} 0} X)^i = X^i$ for
negative integers $i$, otherwise zero. Set ${\tau}_{{\geq} 1} X =
{X/{{\tau}_{{\leq} 0} X}}.$

\begin{prop}\label{15}

For each integer $t \geq 0$, the object $RA_t$ belongs to the
subcategory ${\mathcal{D}}^{{\leq} {t}} \cap \,
^{\perp}{\mathcal{D}}^{{\leq} {-1}} \cap \, {\rm per}A $, and the
object $LA_t$ belongs to the subcategory ${\mathcal{D}}^{{\leq} {0}}
\cap \, ^{\perp}{\mathcal{D}}^{{\leq} {-t-1}} \cap \, {\rm per}A $.
\end{prop}

\begin{proof}
We consider the triangles appearing in the constructions of $RA_t$,
and similarly for $LA_t$.

The object $RA_0 (= P)$ belongs to ${\mathcal{D}}^{{\leq} {0}} \cap
\, ^{\perp}{\mathcal{D}}^{{\leq} {-1}} \cap \, {\rm per}A$ since the
dg algebra $A$ has its homology concentrated in nonpositive degrees.
The object $RA_t$ is an extension of $A^{(t)}$ by ${\Sigma}^{-1}
RA_{t-1}$, which both belong to the subcategory
${\mathcal{D}}^{{\leq} {t}} \cap \, {\rm per}A$. Thus, the object
$RA_t$ belongs to ${\mathcal{D}}^{{\leq} {t}} \cap \, {\rm per}A$.
We do induction on $t$ to show that $RA_t$ belongs to
$^{\perp}{\mathcal{D}}^{{\leq} {-1}}$. Let $Y$ be an object in
${\mathcal {D}}^{\leq {-1}}$. By applying the functor ${\rm
Hom}_{\mathcal {D}} (-,Y)$ to the triangle \[ \xymatrix { RA_t
\ar[r] & A^{(t)} \ar[r]^-{f^{(t)}} & RA_{t-1} \ar[r] & {\Sigma}
RA_t,}
\] we obtain the long exact sequence $$\ldots \ra
{\rm Hom}_{\mathcal {D}} (A^{(t)},Y) \ra {\rm Hom}_{\mathcal {D}} (RA_t,Y)
\ra {\rm Hom}_{\mathcal {D}} ({\Sigma}^{-1} RA_{t-1},Y) \ra \ldots
.$$ Since ${\Sigma} Y$ belongs to ${\mathcal{D}}^{{\leq} {-2}}$, by
hypothesis, the space ${\rm Hom}_{\mathcal {D}} ({\Sigma}^{-1}
RA_{t-1},Y)$ is zero. Thus, the object $RA_t$ belongs to
$^{\perp}{\mathcal{D}}^{{\leq} {-1}}$.
\end{proof}

Assume that $\{e_1,\cdots,e_n\}$ is a collection of primitive
idempotent elements of $H^0A$. We denote by $S_i$ the simple module
corresponding to $e_iA$. For any object $X$ in per$A$, we define the
support of $X$ as follows:

\begin{defn} The {\em support} of $X$ is defined as the set
$$supp\,(X) := \{j \in {\mathbb{Z}} | {\rm Hom}_{{\mathcal {D}}}(X,{\Sigma}^j S_i)
\neq 0 \, \mbox{for some simple module}\, S_i \}.$$
\end{defn}

\begin{prop} \label{27}
For any nonnegative integer $t$, we have the following inclusions:
\begin{itemize}
\item[1)] $\{-t\}  \subseteq supp\,(RA_t) \subseteq [-t,0],$
\item[2)] $\{t\}  \subseteq supp\,(LA_t) \subseteq [0,t]$.
\end{itemize}
\end{prop}

\begin{proof} We only show the first statement, since the second one can be deduced in a similar way.

By Proposition \ref{15}, the object $RA_t$ belongs to
${\mathcal{D}}^{{\leq} {t}} \cap \, ^{\perp}{\mathcal{D}}^{{\leq}
{-1}} \cap \, {\rm per}A $. Therefore, the space ${\rm
Hom}_{\mathcal {D}}(RA_t, {\Sigma}^r S_i)$ vanishes for each integer
$r \geq 1$ since ${\Sigma}^r S_i$ lies in ${\mathcal {D}}^{\leq -1}$
and the space ${\rm Hom}_{\mathcal {D}}(RA_t, {\Sigma}^{r'} S_i)$
vanishes for each integer $r' \leq -t-1$ since ${\Sigma}^{r'} S_i$
lies in ${\mathcal {D}}^{\geq t+1}$. Thus, we have the inclusion
$supp\,(RA_t) \subseteq [-t,0]$.

Let $S_P$ be the simple module corresponding to the
$P$-indecomposable $P$ from which $RA_t$ and $LA_t$ are obtained by
mutation. We will show that ${\rm Hom}_{\mathcal {D}}(RA_t,
{\Sigma}^{-t} S_P)$ is nonzero. Clearly, the space ${\rm
Hom}_{\mathcal {D}}(P, S_P)$ is nonzero. We do induction on the
integer $t$. Assume that the space ${\rm Hom}_{\mathcal
{D}}(RA_{t-1}, {\Sigma}^{1-t} S_P)$ is nonzero. Applying the functor
${\rm Hom}_{\mathcal {D}}(-,{\Sigma}^{1-t}S_P)$ to the triangle
$$RA_t \ra A^{(t)} \ra RA_{t-1} \ra {\Sigma}RA_t,$$where $A^{(t)}$ belongs to (add$A$)$\setminus$(add$P$), we get the long
exact sequence
$$\cdots \ra ({\Sigma}A^{(t)},{\Sigma}^{1-t}S_P) \ra
({\Sigma}RA_t,{\Sigma}^{1-t}S_P) \ra (RA_{t-1},{\Sigma}^{1-t}S_P)
 \ra (A^{(t)}, {\Sigma}^{1-t}S_P)
 \ra \cdots ,$$ where both the
leftmost term and the rightmost term are zero. Therefore, ${\rm
Hom}_{\mathcal {D}}(RA_t,{\Sigma}^{-t}S_P)$ is nonzero. This
completes the proof.
\end{proof}

Now we deduce the following corollary, which can also be deduced
from Theorem 2.43 in \cite{AI}.

\begin{cor}\label{19}
\begin{itemize}
\item[1)] For any two nonnegative integers $r \neq t$, the object
$RA_r$ is not isomorphic to $RA_t$, and the object $LA_r$ is not
isomorphic to $LA_t$.

\item[2)] For any two positive integers $r$ and $t$, the objects
$RA_r$ and $LA_t$ are not isomorphic.
\end{itemize}
\end{cor}

\begin{proof}
Assume that $r > t \geq 0$. Following Proposition \ref{27}, we have
that $${\rm Hom}_{\mathcal {D}}(RA_r, {\Sigma}^{-r}S_P) \neq 0,
\quad \mbox{while} \quad {\rm Hom}_{\mathcal {D}}(RA_t,
{\Sigma}^{-r}S_P) = 0.$$Thus, the objects $RA_r$ and $RA_t$ are not
isomorphic. Similarly for $LA_r$ and $LA_t$. Also in a similar way,
we can obtain the second statement.
\end{proof}

Combining Theorem \ref{5} with Proposition \ref{27}, we can deduce
the following corollary, which is analogous to Corollary 4.2 of
\cite{IR06}:

\begin{cor}\label{9} For any positive integer $l$, up to isomorphism, the object $M$
admits exactly $2l-1$ complements whose supports are contained in
$[1-l,l-1]$. These give rise to basic silting objects. They are the
indecomposable objects $RA_t$ and $LA_t$ for $0 \leq t < l$.
\end{cor}

\vspace{.3cm}
\section{From silting objects to $m$-cluster tilting objects}

Let $\mathcal{F}$ be the full subcategory ${\mathcal{D}}^{{\leq}
{0}} \cap \, ^{\perp}{\mathcal{D}}^{{\leq} {-m-1}} \cap \, {\rm
per}A $ of $\mathcal {D}$. It is called the fundamental domain in
\cite{GUO}. Following Lemma \ref{24}, the category ${\mathcal
{D}}_{fd}(A)$ is a triangulated thick subcategory of per$A$. The
triangulated quotient category ${\mathcal{C}}_A =\, $
{per$A$}/{${\mathcal{D}}_{fd} (A)$} is called the generalized
$m$-cluster category \cite{GUO}. We denote by $\pi$ the canonical
projection functor from per$A$ to ${\mathcal{C}}_A$.

\begin{prop}[\cite{GUO}]\label{3}
Under Assumptions \ref{23}, the projection functor $\pi: {\rm per}A
\longrightarrow {\mathcal {C}}_A$ induces a $k$-linear equivalence
between $\mathcal {F}$ and ${\mathcal {C}}_A$.
\end{prop}

\begin{thm}[\cite{GUO} Theorem 2.2, \cite{KY09} Theorem 7.21]\label{2}
If $A$ satisfies Assumptions \ref{23}, then

\begin{itemize}

\item[1)] the generalized $m$-cluster category ${\mathcal{C}}_A$ is {\rm Hom}-finite and
$(m+1)$-Calabi-Yau;

\item[2)] the object $T = {\pi} (A)$ is an {\em $m$-cluster
tilting object} in ${\mathcal {C}}_A$, i.e., $${\rm add}T = \{ L \in
{\mathcal {C}}_A | \, {\rm Hom}_{{\mathcal{C}}_A} (T, {\Sigma}^r L)
= 0, \, r = 1,\ldots,m \}.$$
\end{itemize}
\end{thm}

\begin{thm} \label{25}
The image of any silting object under the projection functor $\pi:
{\rm per}A \ra {\mathcal {C}}_A$ is an $m$-cluster tilting object in
${\mathcal {C}}_A$.
\end{thm}

\begin{proof}
Assume that $Z$ is an arbitrary silting object in per$A$. Without
loss of generality, we can assume that $Z$ is a cofibrant dg
$A$-module \cite{Ke94}. We denote by $\Gamma$ the dg endomorphism
algebra ${\rm Hom}_A^{\bullet} (Z,Z)$. Since the spaces ${\rm
Hom}_{\mathcal {D}} (Z, {\Sigma}^i Z)$ are zero for all positive
integers $i$, the dg algebra ${\Gamma}$ has its homology
concentrated in nonpositive degrees. The zeroth homology of
${\Gamma}$ is isomorphic to the space ${\rm Hom}_{\mathcal {D}}
(Z,Z)$ which is finite dimensional by Proposition \ref{4}.

Since $Z$ is a compact generator of ${\mathcal {D}}$, the left
derived functor $F = - \overset{L}{\otimes}_{{\Gamma}}Z
$ is a Morita equivalence \cite{Ke94} from ${\mathcal {D}}(\Gamma)$
to $\mathcal {D}$ which sends ${\Gamma}$ to $Z$. Therefore, the dg
algebra ${\Gamma}$ is also (topologically) homologically smooth and
$(m+2)$-Calabi-Yau. Thus, the generalized $m$-cluster category
${\mathcal {C}}_{\Gamma}$ is well defined. The equivalence $F$ also
induces a triangle equivalence from the generalized $m$-cluster
category ${\mathcal {C}}_{{\Gamma}}$ to ${\mathcal {C}}_A$ which
sends $\pi({\Gamma})$ to $\pi(Z)$. By Theorem \ref{2}, the image
$\pi({\Gamma})$ is an $m$-cluster tilting object in ${\mathcal
{C}}_{{\Gamma}}$. Hence, the image of $Z$ is an $m$-cluster tilting
object in ${\mathcal {C}}_A$.
\end{proof}

In particular, for each nonnegative integer $t$, the images of
$LA_t \oplus M$ and $RA_t \oplus M$ in the generalized $m$-cluster
category ${\mathcal {C}}_A$ are $m$-cluster tilting objects.

\begin{defns}Let $A$ be a dg algebra satisfying Assumptions \ref{23} and ${\mathcal
{C}}_A$ its generalized $m$-cluster category.
\begin{itemize}
\item[a)] An object $X$ in ${\mathcal {C}}_A$ is called an {\em almost complete $m$-cluster tilting
object} if there exists some indecomposable object $X'$ in
${\mathcal {C}}_A \setminus$ ({\rm add}X) such that $X \oplus X'$ is
an $m$-cluster tilting object. Here $X'$ is called a {\em
complement} of $X$. In particular, we call $\pi(M)$ an {\em almost
complete $m$-cluster tilting $P$-object}.
\item[b)] An almost complete $m$-cluster tilting object $Y$ is said to
be {\em liftable} if there exists a basic silting object $Z$ in
per$A$ such the $\pi(Z/Z')$ is isomorphic to $Y$ for some
indecomposable direct summand $Z'$ of $Z$.
\end{itemize}
\end{defns}

\begin{prop}\label{39}
Let $A$ be a $3$-Calabi-Yau dg algebra satisfying Assumptions 2.1.
Then any $(1-)$cluster tilting object in ${\mathcal {C}}_A$ is
induced by a silting object in $\mathcal {F}$ under the canonical
projection $\pi$.
\end{prop}

\begin{proof}
Let $T$ be a cluster tilting object in ${\mathcal {C}}_A$. By
Proposition \ref{3}, we know that there exists an object $Z$ in the
fundamental domain $\mathcal {F}$ such that $\pi(Z) = T$.

First we will claim that $Z$ is a partial silting object, that is,
the spaces ${\rm Hom}_{\mathcal {D}}(Z,{\Sigma}^iZ)$ are zero for
all positive integers $i$. Since $Z$ belongs to $\mathcal {F}$,
clearly these spaces vanish for all integers $i \geq 2$. Consider
the case $i = 1$. The following short exact sequence $$0 \ra {\rm
Ext}^1_{\mathcal {D}}(X,Y) \ra {\rm Ext}^1_{{\mathcal {C}}_A}(X,Y)
\ra D{\rm Ext}^1_{\mathcal {D}}(Y,X) \ra 0$$ was shown to exist in
\cite{Am08} for any objects $X,\,Y$ in $\mathcal {F}$. We specialize
both $X$ and $Y$ to the object $Z$. The middle term in the short
exact sequence is zero since $T$ is a cluster tilting object. Thus,
the object $Z$ is partial silting.

Second we will show that $Z$ generates per$A$. Consider the
following triangle $$A \stackrel{f}\rightarrow Z_0 \rightarrow Y
\rightarrow {\Sigma}A$$ in $\mathcal {D}$, where $f$ is a minimal
left (add$Z$)-approximation in $\mathcal {D}$. It is easy to see
that $Y$ also belongs to $\mathcal {F}$. Therefore, the above
triangle can be viewed as a triangle in ${\mathcal {C}}_A$ with $f$
a minimal left (add$Z$)-approximation in ${\mathcal {C}}_A$.
Applying the functor ${\rm Hom}_{{\mathcal {C}}_A}(-,Z)$ to the
triangle, we get the exact sequence
$${\rm Hom}_{{\mathcal {C}}_A}(Z_0,Z) \rightarrow {\rm
Hom}_{{\mathcal {C}}_A}(A,Z) \rightarrow {\rm Hom}_{{\mathcal
{C}}_A}({\Sigma}^{-1}Y,Z) \rightarrow {\rm Hom}_{{\mathcal
{C}}_A}({\Sigma}^{-1}Z_0,Z)$$Therefore the space ${\rm
Hom}_{{\mathcal {C}}_A}(Y,{\Sigma}Z)$ becomes zero. As a
consequence, $Y$ belongs to add$Z$ in ${\mathcal {C}}_A$. Since both
$Y$ and $Z$ are in $\mathcal {F}$, the object $Y$ also belongs to
add$Z$ in $\mathcal {D}$. Therefore, the dg algebra $A$ belongs to
the subcategory thick$Z$ of per$A$. It follows that $Z$ generates
per$A$.
\end{proof}

\begin{thm}\label{20}
The almost complete $m$-cluster tilting $P$-object $\pi(M)$ has at
least $m+1$ complements in ${\mathcal {C}}_A$.
\end{thm}

\begin{proof}
Following Proposition \ref{15} and Corollary \ref{19}, the pairwise
non isomorphic indecomposable objects $LA_t \, (0 \leq t \leq m)$
belong to the fundamental domain $\mathcal {F}$. Therefore, by
Proposition \ref{3}, the $m+1$ objects $\pi(P),\,\pi(LA_1),\dots ,
\pi(LA_m)$ are indecomposable and pairwise non isomorphic in
${\mathcal {C}}_A$. It follows that $\pi(M)$ has at least $m+1$
complements in ${\mathcal {C}}_A$.
\end{proof}

Let us generalize the above theorem:

\begin{thm} \label{26}
Each liftable almost complete $m$-cluster tilting object has at
least $m+1$ complements in ${\mathcal {C}}_A$.
\end{thm}

\begin{proof}
Let $Y$ be a liftable almost complete $m$-cluster tilting object. By
definition there exists a basic silting object $Z$ (assume that $Z$
is cofibrant) in per$A$ such the $\pi(Z/Z')$ is isomorphic to $Y$
for some indecomposable direct summand $Z'$ of $Z$. Let $\Gamma$ be
the dg endomorphism algebra ${\rm Hom}_A^{\bullet} (Z,Z)$. Then
$H^0{\Gamma}$ is a basic algebra.

Similarly as in the proof of Theorem \ref{25}, the dg algebra
$\Gamma$ satisfies Assumptions \ref{23}, and the left derived
functor $F := - \overset{L}{\otimes}_{{\Gamma}}{Z}$ induces a
triangle equivalence from ${\mathcal {C}}_{{\Gamma}}$ to ${\mathcal
{C}}_A$ which sends $\pi({\Gamma})$ to $\pi(Z)$. Let $\Gamma'$ be
the object ${\rm Hom}_A^{\bullet} (Z,Z/Z')$ in per$\Gamma$. Then
$\pi(\Gamma')$ is the almost complete $m$-cluster tilting $P$-object
in ${\mathcal {C}}_{\Gamma}$ which corresponds to $Y$ under the
functor $F$. It follows from Theorem \ref{20} that $\pi(\Gamma')$
has at least $m+1$-complements in ${\mathcal {C}}_{\Gamma}$. So does
the liftable almost complete $m$-cluster tilting object $Y$ in
${\mathcal {C}}_A$.
\end{proof}

\begin{rem}\label{28}
Let $\mathcal {T}$ be a Krull-Schmidt {\rm Hom}-finite triangulated
category with a Serre functor. In fact, following \cite{IY08}, one
can get that any almost complete $m$-cluster tilting object $Y$ in
$\mathcal {T}$ has at least $m+1$ complements. Note that the
notation in \cite{GUO} and \cite{IY08} has some differences with
each other, for example, $m$-cluster tilting objects in \cite{GUO}
correspond to $(m+1)$-cluster tilting subcategories (or objects) in
\cite{IY08}. Here we use the same notation as \cite{GUO}. Set
${\mathcal {Y}} = {\rm add}Y$, ${\mathcal {Z}} =
{\cap}^m_{i=1}\,^{\perp}({\Sigma}^i{\mathcal {Y}})$ and ${\mathcal
{U}} = {\mathcal {Z}}/{\mathcal {Y}}$. Let $X$ be an $m$-cluster
tilting object in $\mathcal {T}$ which contains $Y$ as a direct
summand. Set ${\mathcal {X}} = {\rm add}X$. Then by Theorem 4.9 in
\cite{IY08}, the subcategory ${\mathcal {L}} := {\mathcal
{X}}/{\mathcal {Y}}$ is $m$-cluster tilting in the triangulated
category $\mathcal {U}$. The subcategories ${\mathcal
{L}},\,{\mathcal {L}}\langle 1 \rangle, \, \ldots, \, {\mathcal
{L}}\langle m \rangle$ are distinct $m$-cluster tilting
subcategories of $\mathcal {U}$, where $\langle 1 \rangle$ is the
shift functor in the triangulated category $\mathcal {U}$. Also by
the same theorem, the one-one correspondence implies that the number
of $m$-cluster tilting objects of $\mathcal {T}$ containing $Y$ as a
direct summand is at least $m+1$.
\end{rem}



\vspace{.3cm}
\section{Minimal cofibrant resolutions of simple modules\\ for strongly $(m+2)$-Calabi-Yau case}

The well-known Connes long exact sequence (SBI-sequence) for cyclic
homology \cite{Lod} associated to a dg algebra $A$ is as follows
$$\ldots \ra HH_{m+3}(A) \stackrel{I} \ra HC_{m+3}(A) \stackrel{S} \ra HC_{m+1}(A)
\stackrel{B} \ra HH_{m+2}(A) \stackrel{I} \ra \ldots,$$
where $HH_{\ast}(A)$ denotes the Hochschild homology of $A$ and
$HC_{\ast}(A)$ denotes the cyclic homology.

Let $M$ and $N$ be two dg $A$-modules with $M$ in per$A^e$. Then in
${\mathcal {D}}(k)$ we have the isomorphism

$${\rm RHom}_{A^e} ({\rm
RHom}_{A^e} (M, A^e), N) \simeq M \overset{L} \otimes_{A^e} N .$$ An
element $\xi = \sum^s_{i=1} \xi_{1i} \otimes \xi_{2i} \in H^r (M
\overset{L} \otimes_{A^e} N)$ is {\em non-degenerate} if the
corresponding map
$$\xi^{+}: {\rm RHom}_{A^e} (M, A^e)
\ra {\Sigma}^{r}N$$ given by $\xi^{+}(\phi) = \sum^s_{i=1}
(-1)^{|\phi||\xi|} \phi(\xi_{1i})_2 \xi_{2i} \phi(\xi_{1i})_1$ is an
isomorphism. Throughout this article, we write $|\cdot|$ to denote
the degrees.

Let $l$ be a finite dimensional separable $k$-algebra. We fix a
trace $Tr: l \ra k$ and let $\sigma' \otimes \sigma''$ be the
corresponding Casimir element (i.e., $\sigma' \otimes \sigma'' =
\sum \sigma'_i \otimes \sigma''_i$ and $Tr(\sigma'_i \sigma''_j) =
\delta_{ij}$). An {\em augmented dg $l$-algebra} is a dg algebra $A$
equipped with dg $k$-algebra homomorphisms $l
\stackrel{\varsigma}\ra A \stackrel{\epsilon}\ra l $ such that
$\epsilon \varsigma$ is the identity. Following \cite{VDB} we write
$PCAlgc(l)$ for the category of pseudo-compact augmented dg
$l$-algebras satisfying $ker(\epsilon) = coker(\varsigma) = {\rm
rad}A$. When forgetting the grading, rad$A$ is just the Jacobson
radical of the underlying ungraded algebra $A^u := \prod_r A^r$ of
the dg algebra $A=(A^r)_r$.

The SBI-sequence can be extended to the case that $A \in PCAlgc(l)$,
where $HH_{\ast}(A) (= H_{\ast}(A \overset{L}\otimes_{A^e}A))$ is
computed by the pseudo-compact Hochschild complex. For more details,
see section 8 and Appendix B in \cite{VDB}.

\begin{defn} [\cite{VDB}]
An algebra $A \in PCAlgc(l)$ is  {\em strongly ($m+2$)-Calabi-Yau}
if $A$ is topologically homologically smooth and $HC_{m+1}(A)$
contains an element $\eta$ such that $B\eta$ is non-degenerate in
$HH_{m+2}(A)$.
\end{defn}

\begin{thm}[\cite{VDB}] \label{29}
Let $A \in PCAlgc(l)$. Assume that $A = (A^r)_{r \leq 0}$ is
concentrated in nonpositive degrees. Then $A$ is strongly
($m+2$)-Calabi-Yau if and only if there is a quasi-isomorphism
$(\widehat{T_l V}, d) \ra A$ as augmented dg $l$-algebras with $V$
having the following properties
\begin{itemize}
\item[a)] $d(V) \cap V = 0$;

\item[b)] $V = V_c \oplus lz$ with $z$ an $l$-central element of
degree $-m-1$, $V_c$ finite dimensional and concentrated in degrees
$[-m,0]$;

\item[c)] $dz = \sigma' \eta \sigma''$ with $\eta \in V_c
{\otimes}_{l^e} V_c$ non-degenerate and antisymmetric under the flip
$F: v_1 \otimes v_2 \ra (-1)^{|v_1||v_2|}v_2 \otimes v_1$ for any
$v_1$, $v_2$ in $V_c$.
\end{itemize}

\end{thm}

We would like to present the explicit construction of Ginzburg dg
categories in the following straightforward proposition.

\begin{prop} \label{31}
The completed Ginzburg dg category $\widehat{\Gamma}_{m+2} (Q,W)$
associated to a finite graded quiver $Q$ concentrated in degrees
$[-m,0]$ and a reduced superpotential $W$ being a linear combination
of paths of $Q$ of degree $1-m$ and of length at least $3$, is
strongly ($m+2$)-Calabi-Yau.
\end{prop}

\begin{proof}
We only need to check that $\widehat{\Gamma}_{m+2} (Q,W)$ satisfies
the assumptions and condition 2) in Theorem \ref{29} from its
definition.

Let $l$ be the separable $k$-algebra ${\prod}_{i \in Q_0} ke_i$. Let
${\overline{Q}}^G$ be the double quiver obtained from $Q$ by
adjoining opposite arrows $a^{\ast}$ of degree $-m-|a|$ for arrows
$a \in Q_1$. Let ${\widetilde{Q}}^G$ be obtained from
${\overline{Q}}^G$ by adjoining a loop $t_i$ of degree $-m-1$ for
each vertex $i$. Then the completed Ginzburg dg category
$\widehat{\Gamma}_{m+2} (Q,W)$ is the completed path category
$\widehat{T_l({\widetilde{Q}}^G)}$ with the following differential
\begin{center}
$d(a) = 0$, \quad $a \in Q_1$; \\
$d(t_i) = e_i ({\sum}_{a \in Q_1} [a,a^{\ast}]) e_i, \quad i \in
Q_0$; \\
$d(a^{\ast}) = (-1)^{|a|} \frac{\partial W}{\partial a} = (-1)^{|a|}
{\sum}_{p=uav} (-1)^{(|a|+|v|)|u|}v u, \quad a \in Q_1$;
\end{center}
where the sum in the third formula runs over all homogeneous
summands $p=uav$ of $W$.

Thus, the components of $\widehat{{\Gamma}}_{m+2} (Q,W)$ are
concentrated in nonpositive degrees and ${\widehat{\Gamma}}_{m+2}
(Q,W)$ $( = l \oplus {\prod}_{s \geq 1}
({\widetilde{Q}}^G)^{{\otimes}_l s})$ lies in $PCAlgc(l)$.

The differential above which is induced by the reduced
superpotential $W$ satisfies that $d({\widetilde{Q}}^G) \cap
{\widetilde{Q}}^G = 0$. Set $z = {\sum}_{i \in Q_0} t_i$. Then $z$
is an $l$-central element of degree $-m-1$. Clearly,
${\widetilde{Q}}^G = {\overline{Q}}^G \oplus lz$, the double quiver
${\overline{Q}}^G$ is finite and concentrated in degrees $[-m,0]$,
and the element $d(z) = {\sum}_{a \in Q_1} (a a^{\ast} -
(-1)^{|a||a^{\ast}|} a^{\ast} a)$ is antisymmetric under the flip
$F$.

The last step is to show that $\eta := {\sum}_{a \in Q_1}
[a,a^{\ast}]$ is non-degenerate, that is, the corresponding map
$$\eta^+: \, {\rm Hom}_{l^e} ({\overline{Q}}^G,l^e) \lra {\overline{Q}}^G, \quad \phi \ra (-1)^{|\phi||\eta|} \phi (\eta_1)_2 \eta_2 \phi
(\eta_1)_1$$is an isomorphism. Define morphisms $\phi_{\gamma}
(\gamma \in {\overline{Q}}^G) : {\overline{Q}}^G \ra l^e$ as follows
$$\phi_{\gamma} (\alpha) = \delta_{\alpha \gamma} e_{t(\alpha)} \otimes
e_{s(\alpha)}.$$ Then $\{\phi_{\gamma} | \gamma \in {\overline{Q}}^G
\}$ is a basis of the space ${\rm Hom}_{l^e}
({\overline{Q}}^G,l^e)$. Applying the map $\eta^+$, we obtain the
images $\eta^+ (\phi_a) = (-1)^{m|a|} a^{\ast}$ and $\eta^+
(\phi_{a^{\ast}}) = (-1)^{1+|a^{\ast}|^2} a$ for arrows $a \in Q_1$.
Thus, $\{\eta^+ (\phi_{\gamma}) | \gamma \in {\overline{Q}}^G \}$ is
a basis of ${\overline{Q}}^G$. Therefore, the element $\eta$ is
non-degenerate.
\end{proof}

Now we write down the explicit construction of deformed
preprojective dg algebras as described in \cite{VDB}. Let $Q$ be a
finite graded quiver and $L$ the subset of $Q_1$ consisting of all
loops $a$ of odd degree such that $|a| = -m/2$. Let
${\overline{Q}}^V$ be the double quiver obtained from $Q$ by
adjoining opposite arrows $a^{\ast}$ of degree $-m-|a|$ for $a \in
Q_1\setminus L$ and putting $a^{\ast} = a$ without adjoining an
extra arrow for $a \in L$. Let $N$ be the Lie algebra $k
{\overline{Q}}^V / {[k {\overline{Q}}^V, k {\overline{Q}}^V]}$
endowed with the necklace bracket $\{-,- \}$ (cf. \cite{BL},
\cite{Gi01}). Let $W$ be a superpotential which is a linear
combination of homogeneous elements of degree $1-m$ in $N$ and
satisfies $\{W,W\}=0$ (in order to make the differential
well-defined). Let ${\widetilde{Q}}^V$ be obtained from
${\overline{Q}}^V$ by adjoining a loop $t_i$ of degree $-m-1$ for
each vertex $i$. Then the {\em deformed preprojective dg algebra}
$\Pi(Q,m+2,W)$ is the dg algebra $(k{\widetilde{Q}}^V,d)$ with the
differential
\begin{center}
$da= \{W,a\}=(-1)^{(|a|+1)|a^{\ast}|}\frac{\partial W}{\partial
a^{\ast}}= (-1)^{(|a|+1)|a^{\ast}|}\sum_{p=ua^{\ast}v} (-1)^{(|a^{\ast}|+|v|)|u|}v u$; \\
$da^{\ast} = \{W,a^{\ast}\} = (-1)^{|a|+1} \frac{\partial
W}{\partial a} = (-1)^{|a|+1}
{\sum}_{p=uav} (-1)^{(|a|+|v|)|u|}v u$; \\
$dt_i = e_i ({\sum}_{a \in Q_1} [a,a^{\ast}]) e_i$;
\end{center}
where $a \in Q_1$ and $i \in Q_0$. Later we will denote the
homogeneous elements $rvu \, (r \in k)$ appearing in $d \alpha \,
(\alpha \in {\overline{Q}}^V)$ by $y (\alpha, v, u)$.

\begin{rem} \label{32}
As in Proposition \ref{31}, we see that the completed deformed
preprojective dg algebra $\widehat{\Pi}(Q,m+2,W)$ associated to a
finite graded quiver $Q$ concentrated in degrees $[-m,0]$ and a
reduced superpotential $W$ being a linear combination of paths of
${\overline{Q}}^V$ of length at least $3$, is also strongly
($m+2$)-Calabi-Yau.
\end{rem}

Suppose that $-1$ is a square in the field $k$ and denote by
$\sqrt{-1}$ a chosen square root. Then the class of deformed
preprojective dg algebras is strictly greater than the class of
Ginzburg dg categories. Suppose that $Q$ does not contain {\em
special loops} (i.e., loops of odd degree which is equal to $-m/2$).
Then we can easily see that $\Gamma_{m+2}(Q,W) = \Pi (Q,m+2,-W)$.
Otherwise, let $Q^0$ be the subquiver of $Q$ obtained by removing
the special loops. For each special loop $a$ in $Q_1$, we draw a
pair of loops $a'$ and $a''$ which are also special at the same
vertex of $Q^0$. Denote the new quiver by $Q'$. Let $W'$ be the
superpotential obtained from $W$ by replacing each special loop by
the corresponding element $a' + a'' {\sqrt{-1}}$. Now we define a
map $\iota : \Gamma_{m+2}(Q,W) \ra \Pi (Q',m+2,-W')$, which sends
each special loop $a$ of $Q_1$ to the element $a' + a'' {\sqrt{-1}}$
and its dual $a^{\ast}$ to the element $a' - a'' {\sqrt{-1}}$ in
$\Pi (Q',m+2,-W')$, and is the identity on the other arrows of
${\widetilde{Q}}^G$. Then it is not hard to check that $\iota$ is a
dg algebra isomorphism. It follows that Ginzburg dg categories are
deformed preprojective dg algebras. For the strictness, see the
following example.

\begin{example}\label{36}
Suppose that $m$ is $2$. Let $Q$ be the quiver consisting of only
one vertex `$\bullet$' and one loop $a$ of degree $-1$. Then the
Ginzburg dg category $ {\Gamma}_4(Q,0)$ and the deformed
preprojective dg algebra ${\Pi}(Q,4,0)$ respectively have the the
following underlying graded quivers
\[
{\widetilde{Q}}^G: \quad   \xymatrix  { \bullet \ar@(ur,lu)[]_{a}
\ar@(ur,rd)[]^{{a}^{\ast}} \ar@(ul,ld)[]_{t} },  \quad \quad \quad
{\widetilde{Q}}^V: \quad \xymatrix { \bullet \ar@(ur,lu)[]_{a =
{a}^{\ast}} \ar@(ul,ld)[]_{t} }
\]
where $|a| = |a^{\ast}| = -1$ and $|t| = -3$. The differential takes
the following values $$d(a) = 0 = d(a^{\ast}), \quad \,
d_{\Gamma_{4}(Q,0)}(t) = a a^{\ast}+a^{\ast}a, \quad
d_{\Pi(Q,4,0)}(t) = 2a^2.$$ Then ${\rm dim} H^{-1} (\Gamma_4(Q,0)) =
2$ while ${\rm dim} H^{-1} (\Pi(Q,4,0)) = 1$. Hence, these two dg
algebras are not quasi-isomorphic. Moreover, it is obvious that the
dg algebra $\Pi(Q,4,0)$ can not be realized as a Ginzburg dg
category.
\end{example}

\begin{lem}
Let $\Pi = \widehat{\Pi}(Q,m+2,W)$ be a completed deformed
preprojective dg algebra. Let $x$ (resp. $y$) denote the minimal
(resp. maximal) degree of the arrows of ${\overline{Q}}^V$. Then
there exist a canonical completed deformed preprojective dg algebra
$\Pi' = \widehat{\Pi}(Q',m+2,W')$ isomorphic to $\Pi$ as a dg
algebra, where the quiver $Q'$ is concentrated in degrees
$[-m/2,y]$.
\end{lem}

\begin{proof}
We can construct directly a quiver $Q'$ and a superpotential $W'$.

We claim first that $x+y=-m$. Let $x_1$ (resp. $y_1$) denote the
minimal (resp. maximal) degree of the arrows of $Q$. Then
${\overline{Q}}^V \setminus Q$ is concentrated in degrees
$[-m-y_1,-m-x_1]$. If $x_1 \leq -m-y_1$, then $x= x_1$ and $y_1 \leq
-m-x_1$. Hence, $x+y = x_1 + (-m-x_1) = -m$. Similarly for the case
`$-m-y_1 \leq x_1$'.

Let $Q^0$ be the subquiver of $Q$ which has the same vertices as $Q$
and whose arrows are those of $Q$ with degree belonging to
$[-m/2,y]\, ( = [(x+y)/2,y])$. In this case $|a^{\ast}| = -m - |a|
\in [-m-y,-m/2] = [x,-m/2]$. For each arrow $b$ of $Q$ whose dual
$b^{\ast}$ has degree in $(-m/2,y]$, we add a corresponding arrow
$b'$ to $Q^0$ with the same degree as $b^{\ast}$. Denote the new
quiver by $Q'$. Therefore, the quiver $Q'$ has arrow set
$$\{ a \in Q_1 | \, |a| \in [-m/2,y] \} \cup \{ b'\, | \,
|b'| = |b^{\ast}|, b \in Q_1 \, \mbox{and} \, |b^{\ast}| \in (-m/2,y] \}.$$

We define a map $\iota: {\widetilde{Q}}^V \ra {\widetilde{Q'}}^V$ by
setting $$\iota(a) = a, \, \iota(a^{\ast}) = a^{\ast}; \quad
\iota(t_i) = t_i; \quad  \iota(b) = (-1)^{|b||{b^{\ast}}|+1}
{b'}^{\ast}, \, \iota(b^{\ast}) = b'.$$ Let $W'$ be the
superpotential obtained from $W$ by replacing each arrow $\alpha$ in
$W$ by $\iota(\alpha)$. Then it is not hard to check that the map
$\iota$ can be extended to a dg algebra isomorphism from $\Pi$ to
$\Pi'$.
\end{proof}

In particular, if $Q$ is concentrated in degrees $[-m,0]$, then by
the above lemma, the new quiver $Q'$ is concentrated in degrees
$[-m/2,0]$. If the following two conditions
\begin{itemize}
\item[V1)] $Q$ a finite graded quiver concentrated in degrees
$[-m/2,0]$,

\item[V2)] $W$ a reduced superpotential being a linear combination of paths of
${\overline{Q}}^V$ of degree $1-m$ and of length $\geq 3$,
\end{itemize}
hold, then we will say that the completed deformed preprojective dg
algebra $\widehat{\Pi}(Q,m+2,W)$ is good.

\begin{thm}[\cite{VDB}] \label{30}
Let $A$ be a strongly ($m+2$)-Calabi-Yau dg algebra with components
concentrated in degrees $\leq 0$. Suppose that $A$ lies in
$PCAlgc(l)$ for some finite dimensional separable commutative
$k$-algebra $l$. Then $A$ is quasi-isomorphic to some good completed
deformed preprojective dg algebra.
\end{thm}

We consider the strongly ($m+2$)-Calabi-Yau case in this section, by
Theorem \ref{30}, it suffices to consider good completed deformed
preprojective dg algebras $\Pi = \widehat{\Pi}(Q,m+2,W)$. The simple
$\Pi$-module $S_i$ (attached to a vertex $i$ of $Q$) belongs to the
finite dimensional derived category ${\mathcal {D}}_{fd} (\Pi)$,
hence it also belongs to per$\Pi$. We will give a precise
description of the objects $RA_t$ and $LA_t$ obtained from iterated
mutations of a $P$-indecomposable $e_i \Pi$, where $e_i$ is the
primitive idempotent element associated to a vertex $i$ of $Q$.

\begin{defn}[\cite{Pla}]\label{7}
Let $A = (\widehat{kQ},d)$ be a dg algebra, where $Q$ is a finite
graded quiver and $d$ is a differential sending each arrow to a
(possibly infinite) linear combination of paths of length $\geq 1$.
A dg $A$-module $M$ is {\em minimal perfect} if

\begin{itemize}
\item[a)] its underlying graded module is of the form ${\oplus}^N_{j=1}
R_j$, where $R_j$ is a finite direct sum of shifted copies of direct
summands of $A$, and

\item[b)] its differential is of the form
$d_{int}+{\delta}$, where $d_{int}$ is the direct sum of the
differentials of these $R_j \,\, (1 \leq j \leq N)$, and $\delta$,
as a degree $1$ map from ${\oplus}^N_{j=1} R_j$ to itself, is a
strictly upper triangular matrix whose entries are in the ideal
$\mathfrak{m}$ of $A$ generated by the arrows of $Q$.
\end{itemize}
\end{defn}

\begin{lem}[\cite{Pla}]\label{34}
Let $M$ be a dg $A (= (\widehat{kQ},d))$-module such that $M$ lies
in {\rm per}$A$. Then $M$ is quasi-isomorphic to a minimal perfect
dg $A$-module.
\end{lem}

In the second part of this section, we illustrate how to obtain
minimal perfect dg modules which are quasi-isomorphic to simple
$\Pi$-modules from cofibrant resolutions \cite{KY09}. If a cofibrant
resolution ${\mathbf{p}}X$ of a dg module $X$ is minimal perfect,
then we say ${\mathbf{p}}X$ a {\em minimal cofibrant resolution} of
$X$.

Let $i$ be a vertex of $Q$ and $P_i = e_i {\Pi}$. Consider the short
exact sequence in the category ${\mathcal {C}} (\Pi)$ of dg modules
$$0 \ra ker(p) \stackrel{\iota} \lra P_{i} \stackrel{p}
\lra S_{i} \ra 0 ,$$ where in the category Grmod($\Pi$) of graded
modules $ker(p)$ is the direct sum of $\rho P_{s(\rho)}$ over all
arrows $\rho \in {{\widetilde{Q}}^V}_1$ with $t(\rho) = i$. The
simple module $S_{i}$ is quasi-isomorphic to $cone(ker(p)
\stackrel{\iota} \ra P_{i})$, {\em i.e.}, the dg module
\begin{displaymath}
{X = ( \underline{X} = P_{i} \oplus {\Sigma}X'_0 \oplus \ldots
\oplus {\Sigma} X'_{m+1}, d_X = \left( \begin{array}{cc} d_{P_{i}} &
\iota \\ 0 & -d_{ker(p)}
\end{array} \right)),}
\end{displaymath}
where for each integer $0 \leq j \leq m+1$, the object $X'_j$ is the
direct sum of $\rho P_{s(\rho)}$ ranging over all arrows $\rho \in
{\widetilde{Q}}^V_1$ with $t(\rho) = i$ and $|\rho| = -j$. By
Section 2.14 in \cite{KY09}, the dg module $X$ is a cofibrant
resolution of the simple module $S_{i}$.

Now let $P'_j \, (0 \leq j \leq m+1)$ be the direct sum of
$P_{s(\rho)}$ where $\rho$ ranges over all arrows in
${\widetilde{Q}}^V_1$ satisfying $t(\rho) = i$ and $|\rho| = -j$.
Clearly, $P'_{m+1} = P_{i}$. We require that the ordering of direct
summands $P_{s(\rho)}$ in $P'_j$ is the same as the ordering of
direct summands $\rho P_{s(\rho)}$ in $X'_j$ for each integer $0
\leq j \leq m+1$. Let $Y$ be an object whose underlying graded
module is $\underline{Y} = P_{i} \oplus {\Sigma}P'_0 \oplus
{\Sigma}^2 P'_1 \oplus \ldots \oplus {\Sigma}^{m+2} P'_{m+1}.$ We
endow $\underline{Y}$ with the degree $1$ graded endomorphism
$d_{int} + {\delta}_Y$, where $d_{int}$ is the same notation as in
Definition \ref{7}. The columns of ${\delta}_Y$ have the following
two types: $(\alpha, 0, \ldots, -y_{red}(\alpha,v,u), \ldots, 0)^t$,
and $(t_i, \ldots, -a^{\ast}, \ldots, (-1)^{|b| |b^{\ast}|} b,
\ldots, 0)^t$ for the last column.
Here $\alpha$ is an arrow in ${\overline{Q}}^V$, while $a$ is an
arrow in $Q$ and $b$ is an arrow in ${\overline{Q}}^V \setminus Q$.
Here $y_{red}(\alpha,v,u)$ is obtained from the path $y(\alpha,v,u)
= \beta_s \ldots \beta_1$ (this notation is defined just before
Remark \ref{32}) by removing the factor $\beta_s$. The ordering of
the elements in each column is determined by the ordering of $Y$.

Let $f: Y \ra X$ be a map constructed as the diagonal matrix whose
elements are all arrows in ${\widetilde{Q}}^V_1$ with target at $i$,
together with $e_{i}$ as the first element. Moreover, we require
that the ordering of these arrows is determined by $Y$ (hence also
by $X$), that is, the components of $f$ are of the form
\begin{center} $f_{\rho} : {\Sigma}^{|\rho| + 1} P_{s(\rho)} \lra {\Sigma}
{\rho} P_{s(\rho)}, \quad \quad u \mapsto \rho u .$
\end{center}
It is not hard to check the identity $f (d_{int} + {\delta}_Y) = d_X
f.$ Hence, the morphism $f$ is an isomorphism in ${{\mathcal
{C}}(\Pi)}$, and the map $d_{int} + {\delta}_Y$ makes the object $Y$
into a dg module which is minimal perfect. Therefore, the dg module
$Y$ is a minimal cofibrant resolution of the simple module $S_i$.

\smallskip

In the third part of this section, we show that when there are no
loops of $Q$ at vertex $i$, the truncations of the minimal cofibrant
resolution $Y$ of the simple module $S_i$ produce $RA_t$ and $LA_t
\, (0 \leq t \leq m+1)$ obtained from the $P$-indecomposable $P_i$
by iterated mutations. If we write $M$ for the dg module
${\Pi}/{P_i}$, then the dg modules $P'_j \, (0 \leq j \leq m)$
appearing in $Y$ lie in add$M$. Let ${\varepsilon}_{\leq t}Y$ be the
submodule of $Y$ with the inherited differential whose underlying
graded module is the direct sum of those summands of $Y$ with copies
of shift $\leq t$. Let ${\varepsilon}_{\geq t+1}Y$ be the quotient
module $Y/({\varepsilon}_{\leq t}Y)$. Notice that
${\varepsilon}_{\leq t}Y$ is a truncation of $Y$ for the canonical
weight structure on per$\Pi$, cf. Bondarko, Keller-Nicolas.

\begin{prop} \label{11}
Let $\Pi$ be a good completed deformed preprojective dg algebra
$\widehat{\Pi}(Q,m+2,W)$ and $i$ a vertex of $Q$. Assume that there
are no loops of $Q$ at vertex $i$. Then the following two
isomorphisms
$${\Sigma}^{-t}{\varepsilon}_{\leq t}Y \simeq RA_t \quad
\mbox{and}\quad {\Sigma}^{-t-1}{\varepsilon}_{\geq t+1}Y \simeq
LA_{m+1-t}$$ hold in the derived category $\mathcal {D} := {\mathcal
{D}}{(\Pi)}$ for each integer $0 \leq t \leq m+1$.
\end{prop}

\begin{proof}
We only consider the first isomorphism. Then the second one can be
obtained dually. For arrows of ${\overline{Q}}^V$ of degree $-j$
ending at vertex $i$, we write ${\alpha}_j$; for the symbols
$-y_{red}(\alpha,v,u)$ of degree $-j$, we simply write $-y^j_{red}$,
and for morphisms $f$ of degree $-j$, we write $f_j$, where $0 \leq
j \leq m$. Moreover, we use the notation $[x]$ to denote a matrix
whose entries $x$ have the same `type' (in some obvious sense).

Clearly, when $t=0$, we have that ${\varepsilon}_{\leq 0}Y = P_i =
RA_0$.

When $t=1$, we have the following isomorphisms
\begin{displaymath}
{{\Sigma}^{-1}{\varepsilon}_{\leq 1}Y \simeq ({\Sigma}^{-1}P_i
\oplus P'_0, \left(\begin{array}{cc} d_{{\Sigma}^{-1}P_i} &
-[{\alpha}_0]
\\ 0 & d_{P'_0}
\end{array} \right) ) \simeq {\Sigma}^{-1}cone(P'_0 \overset{h^{(1)}} \lra P_i), }
\end{displaymath}
where each component of $h^{(1)} (= [{\alpha}_0] )$ is the left
multiplication by some ${\alpha}_0$. Since $W$ is reduced, the left
multiplication by ${\alpha}_0$ is nonzero in the space ${\rm
Hom}_{\mathcal {D}}(P'_0,P_i)$. Moreover, only the trivial paths
$e_i$ have zero degree, and there are no loops of ${\overline{Q}}^V$
of degree zero at vertex $i$. It follows that $h^{(1)}$ is a minimal
right (add$M$)-approximation of $P_i$. Then
${\Sigma}^{-1}{\varepsilon}_{\leq 1}Y$ and $RA_1$ are isomorphic in
$\mathcal {D}$.

In general, assume that ${\Sigma}^{-t}{\varepsilon}_{\leq t}Y \simeq
RA_t \, (1 \leq t \leq m)$. We will show that
${\Sigma}^{-t-1}{\varepsilon}_{\leq t+1}Y \simeq RA_{t+1}$. First we
have the following isomorphism
$${\Sigma}^{-t-1}{\varepsilon}_{\leq t+1}Y \simeq ({\Sigma}^{-t-1}P_i
\oplus {\Sigma}^{-t}P'_0 \oplus \ldots \oplus P'_t,$$
\begin{displaymath}
\left(\begin{array}{ccccc} d_{{\Sigma}^{-t-1}P_i} &
(-1)^{t+1}[{\alpha}_0] & \ldots & (-1)^{t+1}[{\alpha}_{t-1}]& (-1)^{t+1}[{\alpha}_t] \\
0 & d_{{\Sigma}^{-t}P'_0} & \ldots & (-1)^t [y^{t-2}_{red}] &(-1)^t
[y^{t-1}_{red}] \\ & \ldots & & \ldots &
\\ 0& 0& \ldots & d_{{\Sigma}^{-1}P'_{t-1}} & (-1)^t [y^0_{red}] \\
0 & 0 & \ldots & 0 & d_{P'_t}
\end{array} \right) )
\end{displaymath}
$$ \simeq {\Sigma}^{-1}cone(P'_t \overset{h^{(t+1)}} \lra RA_t),
\quad \quad \quad \quad \quad \quad \quad \quad \quad \quad \quad$$
where $h^{(t+1)} = ((-1)^t[{\alpha}_t], (-1)^{t-1} [y^{t-1}_{red}],
\ldots, (-1)^{t-1} [y^0_{red}])^t$.
Each component of $h^{(t+1)}$ is a nonzero morphism in ${\rm
Hom}_{\mathcal {D}}(P'_t,RA_t)$, since the superpotential $W$ is
reduced. Otherwise, the arrow ${\alpha}_t$ will be a linear
combination of paths of length $\geq 2$. It follows that $h^{(t+1)}$
is right minimal. Let $L$ be an arbitrary indecomposable object in
add$A$ and $f = (f_t, [f_{t-1}], \ldots, [f_1], [f_0])^t$ an
arbitrary morphism in ${\rm Hom}_{\mathcal {D}}(L,RA_t)$. Then the
vanishing of $d(f)$
implies that $d(f_t) = -[{\alpha}_0][f_{t-1}] - \ldots -
[{\alpha}_{t-2}][f_1] - [{\alpha}_{t-1}][f_0]$. Since there are no
loops of ${\overline{Q}}^V$ of degree $-t$ at vertex $i$, the map
$f_t$ which is homogeneous of degree $-t$ is a linear combination of
the following forms:

$(i)\,f_t= {\alpha}_t g_0$, where $|g_0| = 0$. In this case, the
differential
$$d(f_t) = d(\alpha_tg_0) = d(\alpha_t)g_0 =
[\alpha_0][y^{t-1}_{red}]g_0+\ldots+[\alpha_{t-1}][y^0_{red}]g_0,$$
which implies that $[f_r]$ is equal to $-[y^r_{red}]g_0 \,(0 \leq r
\leq t-1)$. Then the equalities
\begin{displaymath} f =
\left(\begin{array}{c} f_t \\ {[f_{t-1}]} \\ \ldots \\ {[f_1]} \\
{[f_0]}
\end{array}\right) = \left(\begin{array}{c} \alpha_tg_0 \\
-[y^{t-1}_{red}]g_0 \\ \ldots \\ -[y^1_{red}]g_0 \\
-[y^0_{red}]g_0
\end{array}\right) = \left(\begin{array}{c} (-1)^t \alpha_t \\
(-1)^{t-1}[y^{t-1}_{red}] \\ \ldots \\ (-1)^{t-1}[y^1_{red}] \\
(-1)^{t-1}[y^0_{red}]
\end{array}\right)(-1)^tg_0.
\end{displaymath}
hold. Thus, the morphism $f$ factors through $h^{(t+1)}$.

$(ii) \, f_t = \alpha_r g_{t-r}$, where $|g_{t-r}| = r-t \, (0 \leq
r \leq t-1)$. In these cases, the differentials
\begin{center}
$d(f_t) = d(\alpha_r)g_{t-r} + (-1)^r\alpha_rd(g_{t-r}) =
[\alpha_0][y^{r-1}_{red}]g_{t-r} + \ldots
+[\alpha_{r-1}][y^0_{red}]g_{t-r}+(-1)^r\alpha_rd(g_{t-r}),$
\end{center} which implies that $[f_{t-1}]=-[y^{r-1}_{red}]g_{t-r},
\ldots, [f_{t-r}]=-[y^0_{red}]g_{t-r}$ and
$[f_{t-r-1}]=(-1)^{r+1}d(g_{t-r}).$ Then we have that
\begin{displaymath}
\left(\begin{array}{c} f_t \\
{[f_{t-1}]} \\ \ldots \\ {[f_1]} \\ {[f_0]}
\end{array}\right) = \left(\begin{array}{c} \alpha_rg_{t-r} \\ -[y^{r-1}_{red}]g_{t-r} \\
 \ldots \\ -[y^{0}_{red}]g_{t-r} \\ (-1)^{r+1}d(g_{t-r})\\
0 \\ \ldots \\0
\end{array}\right) = d_{RA_t}\left(\begin{array}{c} 0 \\ 0 \\ \ldots \\ 0 \\ (-1)^tg_{t-r}\\0 \\ \ldots \\ 0
\end{array}\right)
 + \left(\begin{array}{c} 0 \\ 0 \\ \ldots \\ 0 \\ (-1)^tg_{t-r}\\0 \\ \ldots \\ 0
\end{array}\right) d_L
\end{displaymath}
is a zero element in ${\rm Hom}_{\mathcal {D}} (L, RA_t)$.
Therefore, the morphism $h^{(t+1)}$ is a minimal right
(add$A$)-approximation of $RA_t (1 \leq t \leq m)$. Hence, the
isomorphism ${\Sigma}^{-t-1}\varepsilon_{\leq t+1}Y \simeq RA_{t+1}$
holds.
\end{proof}

We further assume that the zeroth homology $H^0 \Pi$ is finite
dimensional. Then the dg algebra $\Pi$ satisfies Assumptions
\ref{23} and moreover it is strongly ($m+2$)-Calabi-Yau.

Since the simple module $S_i$ is zero in the generalized $m$-cluster
category ${\mathcal {C}}_{\Pi} = {\rm per}\Pi/{{\mathcal
{D}}_{fd}(\Pi)}$, the corresponding minimal cofibrant resolution $Y$
also becomes zero in ${\mathcal {C}}_{\Pi}$. Taking truncations of
$Y$, we obtain $m+2$ triangles in ${\mathcal {C}}_{\Pi}$
$$\pi(\varepsilon_{\leq t}Y) \lra 0 \lra \pi(\varepsilon_{\geq t+1}Y) \lra \Sigma
\pi(\varepsilon_{\leq t}Y), \quad \quad 0 \leq t \leq m+1,$$where
$\pi: {\rm per}{\Pi} \ra {\mathcal {C}}_{\Pi}$ is the canonical
projection functor. Therefore, the following theorem holds:

\begin{thm} \label{13}
Under the assumptions in Proposition \ref{11} and the assumption
that $H^0 \Pi$ is finite dimensional, the image of $RA_t$ is
isomorphic to the image of $LA_{m+1-t}$ in the generalized
$m$-cluster category ${\mathcal {C}}_{\Pi}$ for each integer $0 \leq
t \leq m+1$.
\end{thm}

\begin{proof} The following isomorphisms
$$\pi(RA_t) \simeq \pi({\Sigma}^{-t}\varepsilon_{\leq t}Y) \simeq
\pi({\Sigma}^{-t-1}\varepsilon_{\geq t+1}Y) \simeq \pi(LA_{m+1-t})$$
are true in ${\mathcal {C}}_{\Pi}$ for all integers $0 \leq t \leq
m+1$.
\end{proof}

In the presence of loops, the objects $RA_t$ and $LA_r$ do not
always satisfy the relations in Theorem \ref{13}. See the following
example.

\begin{example}\label{46}
Suppose that $m$ is $2$. Let $Q$ be the quiver whose vertex set
$Q_0$ has only one vertex `$\bullet$' and whose arrow set $Q_1$ has
two loops $\alpha$ and $\beta$ of degree $-1$. Then the completed
deformed preprojective dg algebra ${\Pi} = \widehat{\Pi}(Q,4,0)$ has
the underlying graded quiver as follows
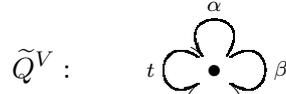
\[
{\widetilde{Q}}^V: \quad\quad   \xymatrix  { \bullet
\ar@(ur,lu)[]_{\alpha} \ar@(ur,rd)[]^{\beta} \ar@(ul,ld)[]_{t} }
\]
with $|\alpha| = |\beta| = -1$ and $|t| = -3$. The differential
takes the following values $$d(\alpha) = 0 = d(\beta), \quad \, d(t)
= 2\alpha^2+2\beta^2.$$

The algebra ${\Pi}$ is an indecomposable object in the derived
category ${\mathcal {D}}(\Pi)$. Let $P = {\Pi}$. Then we have the
equality ${\Pi} = P \oplus M$, where $M=0$. Then $LA_r$ is
isomorphic to ${\Sigma}^rP$ and $RA_r$ is isomorphic to
${\Sigma}^{-r}P$ for all $r \geq 0$.

The zeroth homology $H^0{\Pi}$ is one-dimensional and generated by
the trivial path $e_{\bullet}$. Let ${\mathcal {C}}_{{\Pi}}$ be the
generalized $2$-cluster category. We claim that the image of $RA_1$
in ${\mathcal {C}}_{{\Pi}}$ is not isomorphic to the image of
$LA_2$. Otherwise, assume that $\pi(RA_1)$ is isomorphic to
$\pi(LA_2)$. Then the following isomorphisms hold $${\rm
Hom}_{{\mathcal {C}}_{{\Pi}}}(\pi(LA_2),{\Sigma}\pi(LA_2)) \simeq
{\rm Hom}_{{\mathcal {C}}_{{\Pi}}}(\pi(LA_2),{\Sigma}\pi(RA_1))
\simeq {\rm Hom}_{{\mathcal
{C}}_{{\Pi}}}({\Sigma}^2P,{\Sigma}\pi({\Sigma}^{-1}P)) $$ $$ \simeq
{\rm Hom}_{{\mathcal {C}}_{{\Pi}}}({\Sigma}^2P,P) \simeq {\rm
Hom}_{{\mathcal {D}}(\Pi)}({\Sigma}^2P,P) \simeq H^{-2}{\Pi}.$$ The
left end term of these isomorphisms vanishes since $\pi(LA_2)$ is a
$2$-cluster tilting object, while the right end term is a
$3$-dimensional space whose basis is
$\{\alpha^2,\,\alpha\beta,\,\beta\alpha\}$. Therefore, we obtain a
contradiction.
\end{example}

\vspace{.3cm}

\section{Periodicity property}

\begin{lem}\label{18}
Let $A$ be a dg algebra satisfying Assumptions \ref{23}. Let $x$ and
$y$ be two integers satisfying $x \leq y+m+1$. Suppose that the
object $X$ lies in ${\mathcal {D}}^{\leq x} \cap {\rm per}A$ and the
object $Y$ lies in $^{\perp}{\mathcal {D}}^{\leq y} \cap {\rm
per}A$. Then the quotient functor $\pi: {\rm per}A \ra {\mathcal
{C}}_A$ induces an isomorphism
$${\rm Hom}_{\mathcal {D}} (X, Y) \simeq {\rm Hom}_{{\mathcal {C}}_A}
(\pi (X),\pi (Y)).$$
\end{lem}

\begin{proof}
This proof is quite similar to the proof of Lemma 2.9 given in
\cite{Pla}.

First, we show the injectivity.

Assume that $f: X \ra Y$ is a morphism in $\mathcal {D}$ whose image
in ${\mathcal {C}}_A$ is zero. It follows that $f$ factors through
some $N$ in ${\mathcal {D}}_{fd} (A)$. Let $f = hg$. Consider the
following diagram
\[
\xymatrix{ & X \ar@{.>}[dl] \ar[dr]^g \ar[rr]^f & & Y & & & \\
{\tau}_{\leq x}N \ar[rr] & & N \ar[rr] \ar[ur]^h &  & {\tau}_{\geq
x+1}N \ar[r] & {\Sigma} ({\tau}_{\leq x}N) .}
\] We have that $g$ factors through ${\tau}_{\leq x}N$ because $X \in {\mathcal {D}}^{\leq
x}$ and ${\rm Hom}_{\mathcal {D}}({\mathcal {D}}^{\leq x},
{\tau}_{\geq x+1}N)$ vanishes.

Now since ${\tau}_{\leq x}N$ is still in ${\mathcal {D}}_{fd} (A)$,
by the Calabi-Yau property, the following isomorphism $$D {\rm
Hom}_{\mathcal {D}} ({\tau}_{\leq x}N,Y) \simeq {\rm Hom}_{\mathcal
{D}} (Y, {\Sigma}^{m+2} ({\tau}_{\leq x}N))$$ holds. Since
${\Sigma}^{m+2}({\tau}_{\leq x}N)$ belongs to ${\mathcal {D}}^{\leq
x-m-2} (\subseteq {\mathcal {D}}^{\leq y-1})$, the right hand side
of the above isomorphism is zero. Therefore, the morphism $f$ is
zero in the derived category $\mathcal {D}$.

Second, we show the surjectivity.

Consider an arbitrary fraction $s^{-1}f$ in ${\mathcal {C}}_A$
\[
\xymatrix{X \ar[dr]^f & & Y \ar[dl]_s \\ & U \ar[dl]_r & \\ N & &}
\]
where the cone $N$ of $s$ is in ${\mathcal {D}}_{fd} (A)$. Now look
at the following diagram
\[
\xymatrix{ X \ar[dr]_f \ar@{.>}[r]^{w} & Y \ar[d]^s \ar[dr]^v  & & \\
  & U \ar@{.>}[r]_g \ar[d]^r & Z
\ar[d] &  \\ {\tau}_{\leq x}N \ar[r] & N \ar[r]^{{\pi}_{x+1}\quad}
\ar[d]^u & {\tau}_{\geq x+1}N \ar[r] \ar@{.>}[dl]^h &
{\Sigma}({\tau}_{\leq x}N)
\\ & {\Sigma}Y . & & }
\]
By the Calabi-Yau property, the space ${\rm Hom}_{\mathcal {D}}
({\tau}_{\leq x}N, {\Sigma}Y)$ is isomorphic to $D {\rm
Hom}_{\mathcal {D}} (Y, {\Sigma}^{m+1}({\tau}_{\leq x}N))$, which is
zero since $x-m-1 \leq y$. Thus, there exists a morphism $h$ such
that $u=h \circ {\pi}_{x+1}$. Now we embed $h$ into a triangle in
$\mathcal {D}$ as follows
$$Y \stackrel{v} \lra Z \lra {\tau}_{\geq x+1}N \stackrel{h} \lra
{\Sigma}Y.$$ It follows that the morphism $v$ factors through $s$ by
some morphism $g$. Then we can get a new fraction
\[
\xymatrix{ X \ar[dr]_{g \circ f} & & Y \ar[dl]_{v} \\ & Z  & }
\] where the cone of $v$ is ${\tau}_{\geq x+1}N (\in {\mathcal
{D}}_{fd} (A))$. This fraction is equal to the one we start with
because
$${v^{-1}} (g \circ f) = (g \circ s)^{-1} (g \circ f) \sim s^{-1} f.$$
Moreover, since the space ${\rm Hom}_{\mathcal {D}} (X, {\tau}_{\geq
x+1}N)$ vanishes, there exists a morphism $w: X \ra Y$ such that $g
\circ f = v \circ w$. Therefore, the fraction above is exactly the
image of $w$ in ${\rm Hom}_{\mathcal {D}}(X,Y)$ under the quotient
functor $\pi$.
\end{proof}

Note that in the assumptions of the above lemma, we do not
necessarily suppose that the objects $X$ and $Y$ lie in some shifts
of the fundamental domain.

A special case of Lemma \ref{18} is that, if $X$ lies in ${\mathcal
{D}}^{\leq m} \cap {\rm per}A$, then the quotient functor $\pi: {\rm
per}A \ra {\mathcal {C}}_A$ induces an isomorphism
$${\rm Hom}_{\mathcal {D}} (X, RA_t) \simeq {\rm Hom}_{{\mathcal {C}}_A}
(\pi (X),\pi (RA_t))$$ for any nonnegative integer $t$, where $RA_t$
belongs to $^{\perp} {\mathcal {D}}^{\leq -1}$.

\begin{thm} \label{6} Under the assumptions of Theorem \ref{13},
for each positive integer $t$,

{\rm 1)} the image of $RA_t$
is isomorphic to the image of $RA_{t ({\rm mod} \, m+1)}$ in
${\mathcal {C}}_{\Pi}$,

{\rm 2)} the image of $LA_t$
is isomorphic to the image of $LA_{t ({\rm mod} \, m+1)}$ in
${\mathcal {C}}_{\Pi}$.
\end{thm}

\begin{proof}
We only show the first statement. Then the second one can be
obtained similarly.

Following Theorem \ref{13}, the image of $RA_{m+1}$ in ${\mathcal
{C}}_{\Pi}$ is isomorphic to $P$, which is $RA_0$ by definition. Let
us denote `$t \,({\rm mod} \, m+1)$' by $\overline{t}$. We prove the
statement by induction.

Assume that the image of $RA_t$ is isomorphic to the image of
$RA_{\overline{t}}$ in ${\mathcal {C}}_{\Pi}$. Consider the
following two triangles in ${\mathcal {D}}(\Pi)$
$$RA_{t+1} \lra A^{(t+1)} \stackrel{f^{(t+1)}} \lra RA_t \lra
{\Sigma}RA_{t+1},$$ $$RA_{\overline{t+1}} \lra A^{(\overline{t+1})}
\stackrel{f^{(\overline{t+1})}} \lra RA_{\overline{t}} \lra
{\Sigma}RA_{\overline{t+1}},$$and also consider their images in
${\mathcal {C}}_{\Pi}$. By Lemma \ref{18}, the isomorphism
$${\rm Hom}_{{\mathcal {D}}(\Pi)}(L,RA_t) \simeq {\rm Hom}_{{\mathcal
{C}}_{\Pi}}(L,\pi(RA_t))$$ holds for any object $L \in {\rm add}M$
and any nonnegative integer $t$. Hence, the images $\pi(f^{(t+1)})$
and $\pi(f^{(\overline{t+1})})$ are minimal right
(add$M$)-approximations of $\pi(RA_{t})$ and
$\pi(RA_{\overline{t}})$ in ${\mathcal {C}}_{\Pi}$, respectively. By
hypothesis, $\pi(RA_t)$ is isomorphic to $\pi(RA_{\overline{t}})$.
Therefore, the objects $A^{(t+1)}$ and $A^{(\overline{t+1})}$ are
isomorphic, and $\pi(RA_{t+1})$ is isomorphic to
$\pi(RA_{\overline{t+1}})$ in ${\mathcal {C}}_{\Pi}$. This completes
the statement.
\end{proof}

\begin{rem}
Section 10 in \cite{IY08} gave a class of ($2n+1$)-Calabi-Yau (only
for even integers $2n$, not for all integers $m \geq 2$)
triangulated categories (arising from certain Cohen-Macaulay rings)
which contain infinitely many indecomposable $2n$-cluster tilting
objects.
\end{rem}

In the following, for every integer $m \geq 2$, we construct an
($m+1$)-Calabi-Yau triangulated category which contains infinitely
many indecomposable $m$-cluster tilting objects.

When $m = 2$, we use the same quiver $Q$ as in Example \ref{46}.

When $m > 2$, let $Q$ be the quiver consisting of one vertex
$\bullet$ and one loop $\alpha$ of degree $-1$.

Let $\Pi = {\widehat{\Pi}}(Q,m+2,0)$ be the associated completed
deformed preprojective dg algebra. Clearly, $\Pi$ is an
indecomposable object in the derived category ${\mathcal {D}}(\Pi)$,
the zeroth homology $H^0 \Pi$ is one-dimensional and the path
${\alpha}^s$ is a nonzero element in the homology $H^{-s} \Pi$ ($s
\in {\mathbb{N}}^{\ast}$). Let ${\mathcal {C}}_{\Pi}$ be the
generalized $m$-cluster category and $\pi: \mbox{per} \Pi \ra
{\mathcal {C}}_{\Pi}$ the canonical projection functor. Set $P =
\Pi$. Then $\Pi = P \oplus 0$. For each integer $t \geq 0$, the
object $LA_t$ is isomorphic to ${\Sigma}^tP$ and the object $RA_t$
is isomorphic to ${\Sigma}^{-t}P$. Now we claim that
\begin{itemize}
\item[1)] For any two integers $r > t \geq 0$, the object
$\pi(RA_r)$ is not isomorphic to $\pi(RA_t)$ in ${\mathcal
{C}}_{\Pi}$, and the object $\pi(LA_r)$ is not isomorphic to
$\pi(LA_t)$ in ${\mathcal {C}}_{\Pi}$.

\item[2)] For any two integers $r_1, r_2 \geq 0$, the objects
$\pi(RA_{r_1})$ and $\pi(LA_{r_2})$ are not isomorphic in ${\mathcal
{C}}_{\Pi}$.
\end{itemize}

Otherwise, similarly as in Example \ref{46}, the following
contradictions will appear
$$(0 = )\, {\rm Hom}_{{\mathcal {C}}_{\Pi}}(\pi(RA_t), {\Sigma}\pi(RA_t)) = {\rm Hom}_{{\mathcal
{C}}_{{\Pi}}}(\pi(RA_t),{\Sigma}\pi(RA_r))  \simeq {\rm
Hom}_{{\mathcal {C}}_{{\Pi}}}({\Sigma}^{-t}P,{\Sigma}^{1-r}P) $$ $$
\simeq {\rm Hom}_{{\mathcal {C}}_{{\Pi}}}(P,{\Sigma}^{t-r+1}P)
\simeq {\rm Hom}_{{\mathcal {D}}(\Pi)}(P,{\Sigma}^{t-r+1}P) \simeq
H^{t-r+1}{\Pi}\, ( \neq 0);$$ $$(0 = ) \, {\rm Hom}_{{\mathcal
{C}}_{{\Pi}}}(\pi(LA_r),{\Sigma}\pi(LA_r)) ={\rm Hom}_{{\mathcal
{C}}_{{\Pi}}}(\pi(LA_r),{\Sigma}\pi(LA_t))  \simeq {\rm
Hom}_{{\mathcal {C}}_{{\Pi}}}({\Sigma}^rP,{\Sigma}^{t+1}P) $$ $$
\simeq {\rm Hom}_{{\mathcal {C}}_{{\Pi}}}(P,{\Sigma}^{t-r+1}P)
\simeq {\rm Hom}_{{\mathcal {D}}(\Pi)}(P,{\Sigma}^{t-r+1}P) \simeq
H^{t-r+1}{\Pi} \,( \neq 0);$$ $$(0 = ) \,  {\rm Hom}_{{\mathcal
{C}}_{{\Pi}}}(\pi(LA_{r_2}),{\Sigma}\pi(LA_{r_2})) = {\rm
Hom}_{{\mathcal {C}}_{{\Pi}}}(\pi(LA_{r_2}),{\Sigma}\pi(RA_{r_1}))
\simeq {\rm Hom}_{{\mathcal
{C}}_{{\Pi}}}({\Sigma}^{r_2}P,{\Sigma}^{1-{r_1}}P) $$ $$ \simeq {\rm
Hom}_{{\mathcal {C}}_{{\Pi}}}(P,{\Sigma}^{1-{r_1}-{r_2}}P)  \simeq
{\rm Hom}_{{\mathcal {D}}(\Pi)}(P,{\Sigma}^{1-{r_1}-{r_2}}P) \simeq
H^{1-{r_1}-{r_2}}{\Pi} \,( \neq 0);$$ where the left end terms
become zero, the right end terms are nonzero since $t-r+1 \leq 0$
and $1 - {r_1} - {r_2} < 0$, and the isomorphism $${\rm
Hom}_{{\mathcal {C}}_{{\Pi}}}(P,{\Sigma}^{-s}P) \simeq {\rm
Hom}_{{\mathcal {D}}(\Pi)}(P,{\Sigma}^{-s}P)$$ holds for any $s \in
{\mathbb{N}}$ following Lemma \ref{18}.

Therefore, the ($m+1$)-Calabi-Yau triangulated category ${\mathcal
{C}}_{\Pi}$ contains infinitely many $m$-cluster tilting objects,
and the objects $\pi(RA_t)$ and $\pi(LA_r)$ do not satisfy the
relations in Theorem \ref{13} and Theorem \ref{6} in the presence of
loops.

\vspace{.3cm}

\section{AR ($m+3$)-angles related to $P$-indecomposables}

Let $\mathcal {T}$ be an additive Krull-Schmidt category. We denote
by $J_{\mathcal {T}}$ the {\em Jacobson radical} \cite{ARS} of
$\mathcal {T}$.
Let $f \in {\mathcal {T}} (X,Y)$ be a morphism. Then $f$ is called
(in \cite{IY08}) a {\em sink map} of $Y \in {\mathcal {T}}$ if $f$
is right minimal, $f \in J_{\mathcal {T}}$, and $${\mathcal {T}}
(-,X) \stackrel{f \cdot} \lra J_{\mathcal {T}} (-,Y) \lra 0$$ is
exact as functors on $\mathcal {T}$. The definition of {\em source
maps} is given dually.

Let $n$ be a positive integer. Given $n$ triangles in a triangulated
category,
$$X_{i+1} \stackrel{b_{i+1}} \ra B_i \stackrel{a_i}\ra X_i \ra
{\Sigma} X_{i+1}, \quad 0 \leq i < n,$$ the complex
$$X_n \stackrel{b_{n}} \ra
B_{n-1} \stackrel{b_{n-1} a_{n-1}} \lra B_{n-2} \ra \ldots \ra B_1
\stackrel{b_{1} a_1} \lra B_0 \stackrel{a_0} \ra X_0 $$ is called
(in \cite{IY08}) an {\em $(n+2)$-angle}.

\begin{defn}[\cite{IY08}] \label{10}
Let $H$ be an $m$-cluster tilting object in a Krull-Schmidt
triangulated category. We call an $(m+3)$-angle with $X_0, X_{m+1}$
and all $B_i (0 \leq i \leq m)$ in add$H$ an {\em AR $(m+3)$-angle}
if the following conditions are satisfied
\begin{itemize}
\item[a)] $a_0$ is a sink map of $X_0$ in add$H$ and $b_{m+1}$
is a source map of $X_{m+1}$ in add$H$, and
\item[b)] $a_i$ (resp. $b_i$) is a minimal right (resp. left) (add$H$)-approximation of $X_i$ for each integer $1 \leq i \leq m$.
\end{itemize}
\end{defn}

\begin{rem}
An AR $(m+3)$-angle with right term $X_0$ (resp. left term
$X_{m+1}$) depends only on $X_0$ (resp. $X_{m+1}$) and is unique up
to isomorphism as a complex.
\end{rem}

We will use the AR angle theory to show the following theorem, which
gives a more virtual criterion than Theorem 5.8 in \cite{IY08} for
our case.

\begin{thm} \label{16}
Let $\Pi$ be a good completed deformed preprojective dg algebra
$\widehat{\Pi}(Q,m+2,W)$ and $i$ a vertex of $Q$. Assume that the
zeroth homology $H^0 \Pi$ is finite dimensional and there are no
loops of $Q$ at vertex $i$. Then the almost complete $m$-cluster
tilting $P$-object $\Pi/{e_i\Pi}$ has exactly $m+1$ complements in
the generalized $m$-cluster category ${\mathcal {C}}_{\Pi}$.
\end{thm}

\begin{proof}
Set $RA_0 = P_i = e_i\Pi$ and $M = \Pi/{e_i\Pi}$. Section 4 gives us
a construction of iterated mutations $RA_t$ of $P_i$ in the derived
category ${\mathcal {D}}(\Pi)$, that is, the morphism $h^{(1)}: P'_0
\ra P_i$ is a minimal right (add$M$)-approximation of $P_i$, and
morphisms $h^{(t+1)}: P'_t \ra RA_t \, (1 \leq t \leq m)$ are
minimal right (add$A$)-approximations of $RA_{t}$ with $P'_t$ in
add$M$. Let $\mathcal {A}$ (resp. $\mathcal {M}$) denote the
subcategory add$\pi(\Pi)$ (resp. add$\pi(M)$) in the generalized
$m$-cluster category ${\mathcal {C}}_{\Pi}$.

\smallskip

Step 1. Since $P'_0,\,P_i$ and $M$ are in the fundamental domain,
the morphism $h^{(1)}$ can be viewed as a minimal right $\mathcal
{M}$-approximation in ${\mathcal {C}}_{\Pi}$, that is, the sequence
$${\mathcal {A}}(-,P'_0)|_{\mathcal {M}} \stackrel{h^{(1)} \cdot}\lra
{\mathcal {A}}(-,P_i)|_{\mathcal {M}} = J_{\mathcal
{A}}(-,P_i)|_{\mathcal {M}} \ra 0,$$ is exact as functors on
$\mathcal {M}$. Since there are no loops of ${\overline{Q}}^V$ of
degree zero at vertex $i$, the Jacobson radical of ${\rm
End}_{\mathcal {A}}(P_i) \, (\simeq {\rm End}_{{\mathcal
{D}}(\Pi)}(P_i))$ consists of combinations of cyclic paths $p = a_1
\ldots a_r \,(r \geq 2)$ of ${\overline{Q}}^V$ of degree zero. The
path $p$ factors though $e_{s(a_1)}\Pi$ and factors through
$h^{(1)}$. Therefore, we have an exact sequence
$${\mathcal {A}}(P_i, P'_0) \stackrel{h^{(1)} \cdot}\lra {\rm rad} \, {\rm End}_{\mathcal {A}}(P_i) \lra
0.$$ Thus, the morphism $h^{(1)}$ is a sink map in the subcategory
$\mathcal {A}$.

\smallskip

Step 2. The morphisms $h^{(t+1)} \, (1 \leq t \leq m)$ are minimal
right (add$A$)-approximations of $RA_{t}$ with $P'_t$ in add$M$.
Since the objects $RA_t (1 \leq t \leq m)$ and $P'_t$ lie in the
shift ${\Sigma}^{-m}{\mathcal {F}}$ of the fundamental domain by
Proposition \ref{15}, the images of $h^{(t+1)}$ are minimal right
$\mathcal {A}$-approximations in ${\mathcal {C}}_{\Pi}$.

\smallskip

Step 3. Consider the morphisms $\alpha^{(t)}$ in the triangles of
constructing $RA_t$ in ${\mathcal {D}}(\Pi)$
$${\Sigma}^{-1}RA_{t-1} \lra RA_t \stackrel{\alpha^{(t)}}\lra P'_{t-1}
 \stackrel{h^{(t)}} \lra RA_{t-1}, \quad 1 \leq t \leq m.$$
We already know that the maps $\alpha^{(t)}$ are minimal left
(add$M$)-approximations in ${\mathcal {D}}(\Pi)$. Now applying the
functor ${\rm Hom}_{{\mathcal {D}}(\Pi)}(-,P_i)$ to the above
triangles, we obtain long exact sequences $$\ldots \ra {\rm
Hom}_{{\mathcal {D}}(\Pi)}(P'_{t-1},P_i) \lra {\rm Hom}_{{\mathcal
{D}}(\Pi)}(RA_t,P_i) \lra {\rm Hom}_{{\mathcal
{D}}(\Pi)}({\Sigma}^{-1}RA_{t-1},P_i) \ra \ldots.$$The terms ${\rm
Hom}_{{\mathcal {D}}(\Pi)}({\Sigma}^{-1}RA_{t-1},P_i)$ are zero
since all $RA_{t-1}$ lie in $^{\perp}{{\mathcal {D}}(\Pi)}^{\leq
-1}$. Hence, the morphisms $\alpha^{(t)}$ are minimal left
(add$A$)-approximations in ${\mathcal {D}}(\Pi)$. Since the objects
$RA_t (1 \leq t \leq m)$ and $P'_t$ lie in the shift
${\Sigma}^{-m}{\mathcal {F}}$, the images of $\alpha^{(t)}$ are
minimal left $\mathcal {A}$-approximations in ${\mathcal
{C}}_{\Pi}$.

\smallskip

Step 4. Consider the following two triangles in ${\mathcal
{D}}(\Pi)$
$$RA_{m+1} \stackrel{\alpha^{(m+1)}}\lra P'_m \stackrel{h^{(m+1)}} \lra RA_m \lra
{\Sigma}RA_{m+1},$$ $$P_i \stackrel{g^{(1)}} \lra P'_m
\stackrel{\beta^{(1)}}\lra LA_1 \lra {\Sigma}P_i.$$ Since the
objects $P_i,\,P'_m$ and $LA_{1}$ are in the fundamental domain
$\mathcal {F}$, the second triangle can also be viewed as a triangle
in ${\mathcal {C}}_{\Pi}$ and the morphism $\beta^{(1)}$ is a
minimal right $\mathcal {M}$-approximation of $LA_1$. Note that the
objects $RA_m$ and $P'_m$ belong to ${\Sigma}^{-m}{\mathcal {F}}$.
Hence, the image of the first triangle
$$\pi(RA_{m+1}) \stackrel{\pi({\alpha}^{(m+1)})} \lra P'_m
\stackrel{\pi(h^{(m+1)})} \lra \pi(RA_m) \lra
{\Sigma}\pi(RA_{m+1})$$ is a triangle in ${\mathcal {C}}_{\Pi}$ with
$\pi(h^{(m+1)})$ a minimal right ${\mathcal {M}}$-approximation of
$\pi(RA_m)$. By Theorem \ref{13}, the image of $RA_m$ is isomorphic
to the image of $LA_1$ in ${\mathcal {C}}_{\Pi}$. Thus, the images
of these two triangles in ${\mathcal {C}}_{\Pi}$ are isomorphic. We
can also check that $g^{(1)}$ is a source map in $\mathcal {A}$ as
Step 1. Therefore, the image $\pi(\alpha^{(m+1)})$ is also a source
map in $\mathcal {A}$ with $\pi(RA_{m+1})$ isomorphic to $P_i$ in
${\mathcal {C}}_{\Pi}$.

\smallskip

Step 5. Now we form the following $(m+3)$-angle in ${\mathcal
{C}}_{\Pi}$
\[
\xymatrix{P_i = \pi(RA_{m+1}) \ar[r]^{\quad \quad \quad
\varphi_{m+1}} & P'_m \ar[r]^{\varphi_m} & P'_{m-1} \ar[r] & \ldots
\ar[r] & P'_1 \ar[r]^{\varphi_1} & P'_0 \ar[r]^{\varphi_0} & P_i,
\,}
\]
where $\varphi_0$ is equal to $\pi(h^{(1)})$, the morphism
$\varphi_t \, (1 \leq t \leq m)$ is the composition
$\pi(\alpha^{(t)})\pi(h^{(t+1)})$, and $\varphi_{m+1}$ is equal to
$\pi(\alpha^{(m+1)})$. From the above four steps, we know that
$\varphi_0$ is a sink map in $\mathcal {A}$, and $\varphi_{m+1}$ is
a source map in $\mathcal {A}$. Furthermore, this $(m+3)$-angle is
the AR $(m+3)$-angle determined by $P_i$. Since the indecomposable
object $P_i$ does not belong to add$({\oplus}^m_{t=0}P'_t)$,
following Theorem 5.8 in \cite{IY08}, the almost complete
$m$-cluster tilting $P$-object $\Pi/{e_i\Pi}$ has exactly $m+1$
complements $e_i \Pi,\,\pi(RA_1), \ldots, \pi(RA_m)$ in ${\mathcal
{C}}_{\Pi}$. The proof is completed.
\end{proof}

\vspace{.3cm}

\section{Liftable almost complete $m$-cluster tilting objects \\ for strongly ($m+2$)-Calabi-Yau case}
Keep the assumptions as in Theorem \ref{16}. Let $\Pi =
{\widehat{\Pi}(Q,m+2,W)}$. Let $Y$ be a liftable almost complete
$m$-cluster tilting object in the generalized $m$-cluster category
${\mathcal {C}}_{\Pi}$. Assume that $Z$ is a basic cofibrant silting
object in per$\Pi$ such that $\pi (Z/{Z'})$ is isomorphic to $Y$,
where $\pi: {\rm per}\Pi \ra {\mathcal {C}}_{\Pi}$ is the canonical
projection and $Z'$ is an indecomposable direct summand of $Z$. Let
$A$ be the dg endomorphism algebra ${\rm Hom}^{\bullet}_{\Pi}(Z,Z)$
and $F$ the left derived functor $- \overset{L}{\otimes}_A Z$. From
the proof of Theorem \ref{25}, we know that $F$ is a Morita
equivalence from ${\mathcal {D}}(A)$ to ${\mathcal {D}}(\Pi)$ and
$A$ satisfies Assumptions \ref{23}. We denote the truncated dg
subalgebra ${\tau}_{\leq 0}A$ by $E$. Since $A$ has its homology
concentrated in nonpositive degrees, the canonical inclusion $E
\hookrightarrow A$ is a quasi-isomorphism. Then the left derived
functor $- \overset{L}{\otimes}_{E}A$ is a Morita equivalence from
${\mathcal {D}}(E)$ to ${\mathcal {D}}(A)$.

\begin{thm}[\cite{Ke96}] \label{33}
Let $l$ be a commutative ring. Let $B$ and $B'$ be two dg
$l$-algebras and $X$ a dg $B$-$B'$-bimodule which is cofibrant over
$B$. Assume that $B$ and $B'$ are flat as dg $l$-modules and $$-
\overset{L}{\otimes}_{B'} X: {\mathcal {D}}(B') \ra {\mathcal
{D}}(B)$$ is an equivalence. Then the dg algebras $B$ and $B'$ have
isomorphic cyclic homology and isomorphic Hochschild homology.
\end{thm}

A corollary of Theorem \ref{33} is that $B'$ is strongly
($m+2$)-Calabi-Yau if and only if so is $B$.

The object $Z$ is canonically an $k$-module, so the dg algebras $A$
and $E$ are $k$-algebras. Thus, the derived equivalent dg algebras
$\Pi, A$ and $E$ are flat as dg $k$-modules. Following Remark
\ref{32} and Theorem \ref{33}, the dg algebras $A$ and $E$ are also
strongly ($m+2$)-Calabi-Yau.

We will show that the dg algebra $E$ satisfies the assumption in
Theorem \ref{30}, that is $E$ lies in $PCAlgc(l')$ for some finite
dimensional separable commutative $k$-algebra $l'$. In fact, $l' =
{\prod}_{|Z|} k$, where $|Z|$ is the number of indecomposable direct
summands of $Z$ in per$\Pi$. Furthermore, from the following lemma,
we can deduce that $l' = l$.

\begin{lem} \label{38}
Suppose that $B$ is a dg algebra with positive homologies being
zero. Then all basic cofibrant silting objects have the same number
of indecomposable direct summands in ${\rm per}B$.
\end{lem}

\begin{proof}
The triangulated category per$B$ contains an additive subcategory
${\mathcal {B}} :=$ add$B$. Since the dg algebra $B$ has its
homology concentrated in nonpositive degrees, it follows that $${\rm
Hom}_{{\rm per}B} ({\mathcal {B}},{\Sigma}^p{\mathcal {B}}) = 0,
\quad \, p > 0.$$ Since the category per$B$, which consists of the
compact objects in ${\mathcal {D}}(B)$, and the category add$B$ are
both idempotent split, by Proposition 5.3.3 of \cite{Bon}, the
isomorphism
$$K_0 ({\rm per}B) \simeq K_0 ({\rm add}B)$$ holds, where $K_0 (-)$
denotes the Grothendieck group.

Let $Z$ be any basic cofibrant silting object in per$B$ and $B'$ its
dg endomorphism algebra ${\rm Hom}^{\bullet}_B(Z,Z)$. Then $B'$ has
its homology concentrated in nonpositive degrees and per$B'$ is
triangle equivalent to per$B$. Therefore, we have $K_0 ({\rm per}B')
\simeq K_0 ({\rm add}B')$ and $K_0 ({\rm per}B') \simeq K_0 ({\rm
per}B)$. As a consequence, the following isomorphisms hold
$$K_0 ({\rm add}B) \simeq K_0 ({\rm add}B') \simeq K_0 ({\rm
add}Z).$$ Thus, any basic cofibrant silting object in ${\rm per}B$
has the same number of indecomposable direct summands as that of the
dg algebra $B$ itself.
\end{proof}

When forgetting the grading, the dg algebra $E$ becomes to be $E^u
:= Z^0 A \oplus ({\prod}_{r < 0} A^r)$, where $Z^0 A ( = {\rm
Hom}_{{\mathcal {C}}(\Pi)}(Z,Z))$ consists of the zeroth cycles of
$A$. For any $x \in {\prod}_{r < 0}A^r$, the element $1+x$ clearly
has an inverse element. It follows that ${\prod}_{r<0}A^r$ is
contained in rad($E^u$). We have the following canonical short exact
sequence
$$0 \ra B^0A \ra Z^0A \stackrel{p}\ra H^0A \ra 0,$$where $B^0A$ is a
two-sided ideal of the algebra $Z^0A$ consisting of the zeroth
boundaries of $A$.

Following from Lemma \ref{34}, without loss of generality, we can
assume that the basic silting object $Z$ is a minimal perfect dg
$\Pi$-module.

\begin{lem}
Keep the above notation and suppose that $Z$ is a minimal perfect dg
$\Pi$-module. Then $B^0A$ lies in the radical of $Z^0A$.
\end{lem}

\begin{proof}
Let $f$ be an element in $B^0A$. Then $f$ is of the form $d_Z h + h
d_Z$ for some degree $-1$ morphism $h: Z \ra Z$. Since $Z$ is
minimal perfect, the entries of $f$ lie in the ideal $\mathfrak{m}$
generated by the arrows of $\widetilde{Q}^V$. Then for any morphism
$g: Z \ra Z$, the morphism $1_Z - gf$ admits an inverse $1_Z + gf +
(gf)^2 + \ldots$. Similarly for the morphism $1_Z - fg$. It follows
that $f$ lies in the radical of the algebra $Z^0A$. This completes
the proof.
\end{proof}

The epimorphism $p$ in the above short exact sequence induces an
epimorphism $$\overline{p}: Z^0A/{{\rm rad}(Z^0A)} \ra H^0A/{{\rm
rad}(H^0A)}.$$ Since $B^0A$ lies in the radical of $Z^0A$, the
epimorphism $\overline{p}$ is an isomorphism. Therefore, the
following isomorphisms $$E^u/{{\rm rad}(E^u)} \simeq Z^0A/{{\rm
rad}(Z^0A)} \simeq H^0A/{{\rm rad}(H^0A)}$$ are true. Note that
per$\Pi$ is Krull-Schmidt and Hom-finite. Since the algebra $E_i :=
{\rm End}_{{\rm per}\Pi}(Z_i)$ is local and $k$ is algebraically
closed, the quotient $E_i/{{\rm rad}(E_i)}$ is isomorphic to $k$.
Then we have that
$$H^0A/{{\rm rad}(H^0A)} \simeq {\rm End}_{\mbox{per}\Pi}(Z)/{{\rm rad}({\rm End}_{\mbox{per}\Pi}(Z))}
 \simeq {\prod}_{|Z|} E_i/{{\rm rad}(E_i)} \simeq {\prod}_{|Z|}k \, ( = l).$$
Hence, the dg algebra $E$ lies in $PCAlgc(l)$. Therefore, $E$ is
quasi-isomorphic to some good completed deformed preprojective dg
algebra $\widehat{\Pi}(Q',m+2,W')$ (denoted by $\Pi'$). Moreover,
$H^0 \Pi'$ is equal to $H^0A$ which is finite dimensional.

The following diagram

\[
\xymatrix{ {\rm per}\Pi' \ar[r]^{- \overset{L}\otimes_{\Pi'}E}
\ar[d] & {\rm per}E \ar[r]^{- \overset{L}\otimes_{E}A} \ar[d] & {\rm
per}A \ar[r]^{- \overset{L}\otimes_{A}Z} \ar[d] & {\rm per}\Pi
\ar[d] \\
{\mathcal {C}}_{\Pi'} \ar[r] & {\mathcal {C}}_E \ar[r]& {\mathcal
{C}}_A \ar[r] & {\mathcal {C}}_{\Pi}}
\]
is commutative, where each functor in the rows is an equivalence and
each functor in a column is the canonical projection. The preimage
of $Z$ in ${\rm per}\Pi'$, under the equivalence $F$ given by the
composition of the functors in the top row, is $\Pi'$. Let $\Pi'_0 =
e_j \Pi'$ be the $P$-indecomposable dg $\Pi'$-module such that
$F(\Pi'_0) = Z'$ in per$\Pi$, where $j$ is a vertex of $Q'$. Assume
that there are no loops of $Q'$ at vertex $j$. It follows from
Theorem \ref{16} that the almost complete $m$-cluster tilting
$P$-object $\Pi'/{\Pi'_0}$ has exactly $m+1$ complements in
${\mathcal{C}}_{\Pi'}$. Note that the image of $\Pi'/{\Pi'_0}$ in
${\mathcal {C}}_{\Pi}$, under the equivalence given by the
composition of the functors in the bottom row, is $Y$. Therefore,
the liftable almost complete $m$-cluster tilting object $Y$ has
exactly $m+1$ complements in ${\mathcal{C}}_{\Pi}$.

As a conclusion, we write down the following theorem.

\begin{thm} \label{35}
Let $\Pi$ be a good completed deformed preprojective dg algebra
$\widehat{\Pi}(Q,m+2,W)$ whose zeroth homology $H^0 \Pi$ is finite
dimensional. Let $Z$ be a basic silting object in per$\Pi$ which is
minimal perfect and cofibrant. Denote by $E$ the dg algebra
$\tau_{\leq 0}({\rm Hom}^{\bullet}_{\Pi}(Z,Z))$. Then
\begin{itemize}
\item[1)] $E$ is quasi-isomorphic to some good completed deformed preprojective
dg algebra $\Pi' = {\widehat{\Pi}}(Q',m+2,W')$, where the quiver
$Q'$ has the same number of vertices as $Q$ and $H^0 \Pi'$ is finite
dimensional;

\item[2)] let $Y$ be a liftable almost complete $m$-cluster
tilting object of the form $\pi(Z/Z')$ in ${\mathcal{C}}_{\Pi}$ for
some indecomposable direct summand $Z'$ of $Z$. If we further assume
that there are no loops at the vertex $j$ of $Q'$, where $e_j \Pi'
\overset{L} \otimes_{\Pi'} Z = Z'$, then $Y$ has exactly $m+1$
complements in ${\mathcal{C}}_{\Pi}$.
\end{itemize}
\end{thm}

Here we would like to point out a special case of the above theorem,
namely $m=1$ and $Z = LA_1^{(k)}$ with respect to some vertex $k$ of
$Q$. Let $(Q^{\star}, W^{\star})$ denote the (reduced) mutation
${\mu_k(Q,W)}$ defined in \cite{DWZ} of the quiver with potential
$(Q,W)$ at vertex $k$. Let $A$ be the dg endomorphism algebra ${\rm
Hom}_{\Pi}^{\bullet}(Z,Z)$ and $\Pi^{\star}$ the good completed
deformed preprojective dg algebra
${\widehat{\Pi}}(Q^{\star},m+2,W^{\star})$. By \cite{KY09}, there is
a canonical morphism from $\Pi^{\ast}$ to $A$. Define three functors
as follows:
$$F = - \overset{L} \otimes_{\Pi^{\star}} Z, \quad \,  F_1 = - \overset{L} \otimes_{\Pi^{\star}} A, \quad \, F_2 = - \overset{L} \otimes_{A} Z.$$
Clearly, we have that $F = F_2 F_1$ and $F_2$ is a quasi-inverse equivalence. It was shown in \cite{KY09} that $F$ is a quasi-inverse equivalence. The
following isomorphisms
$$H^{n}(\Pi^{\star}) \simeq {\rm Hom}_{\Pi^{\star}} (\Pi^{\star}, \Sigma^n \Pi^{\star}) \simeq {\rm Hom}_A(A, \Sigma^n A) \simeq H^{n}A$$
become true, which implies that $\Pi^{\star}$ and $A$ are quasi-isomorphic. Therefore, the quiver with potential $(Q',W')$ appearing in Theorem \ref{35} 1)
for this special case can be chosen as $\mu_k (Q,W)$.

As the end part of this section, we state a `reasonable' conjecture
about the non-loop assumption in the above theorem for completed
deformed preprojective dg algebras.

\begin{defn}
Let $r$ be a positive integer. An algebra $A \in PCAlgc(l)$ is said
to be {\em $r$-rigid} if
$$HH_0 (A) \simeq l, \quad {\rm {and}} \quad HH_p (A)= 0 \,\, (1
\leq p \leq r-1),$$where $HH_{\ast}(A)$ is the pseudo-compact
version of the Hochschild homology of the dg algebra $A$.
\end{defn}

\begin{rem} For completed Ginzburg algebras associated to quivers with potentials, our definition of $1$-rigidity coincides
with the definition of rigidity in \cite{DWZ}. Proposition 8.1 in
\cite{DWZ} states that any rigid reduced quiver with potential is
$2$-acyclic. Then no loops will be produced following their
mutation rule. Although 
we do not know whether the quiver $Q'$ related to such a silting
object as in Theorem \ref{35} can be obtained from mutation of
quivers with potentials, we can still obtain that the quiver $Q'$
always does not contain loops in the condition of $1$-rigidity (see
Corollary \ref{41}).
\end{rem}

\begin{prop}\label{42}
The completed deformed preprojective dg algebras $\Pi =
{\widehat{\Pi}(Q,m+2,0)}$ associated to acyclic quivers $Q$ are
$m$-rigid.
\end{prop}

\begin{proof}
It is clear that the zeroth component ${\Pi}^0$ of $\Pi$ is just the
finite dimensional path algebra $kQ$ (denoted by $B$) and the
$(-p)$th component of $\Pi$ is zero for $1 \leq p \leq m-1$. Thus,
the Hochschild homology of $\Pi$ is given by
\begin{center}
$HH_0(\Pi) = B/{[B, B]} = {\prod}_{|Q_0|}k,$ \\
\smallskip
$HH_p(\Pi) = HH_p(B) = ker({\partial}^0_p)/
\mbox{Im}({\partial}^0_{p+1}) \quad (1 \leq p \leq m-1),$
\end{center}
where ${\partial}^0_p: B^{{\otimes}_k (p+1)} \ra B^{{\otimes}_k p}$
is the $p$th row differential of the uppermost row in the Hochschild
complex $X := \Pi \overset{L}{\otimes}_{{\Pi}^e} \Pi$.

Since the path algebra $kQ$ is of finite dimension and of finite
global dimension and $k$ is algebraically closed, we have $HH_p(B) =
0$ for all integers $p > 0$, cf. Proposition 2.5 of \cite{Ke96}. It
follows that the dg algebra ${\widehat{\Pi}(Q,m+2,0)}$ is $m$-rigid.
\end{proof}

\begin{prop}\label{40}
Let $\Pi = {\widehat{\Pi}(Q,m+2,W)}$ be a good completed deformed
preprojective dg algebra and $p$ a fixed integer in the segment $[0,
m]$. Suppose the $p$-th Hochschild homology of $\Pi$ satisfies the
isomorphism
\begin{displaymath}
 HH_p (\Pi) \simeq \left\{ \begin{array}{ll}
       {\prod}_{|Q_0|} k & \textrm{if}\,\, p=0,\\
       0 & \textrm{if} \,\, p \neq 0.
\end{array} \right.
\end{displaymath}
Then ${\overline{Q}}^V$ does not contain loops with zero
differential and of degree $-p$.
\end{prop}

\begin{proof}
Let $a$ be a loop of ${\overline{Q}}^V$ at some vertex $i$ with zero
differential and of degree $-p$. The element $a$ lies in the
rightmost column of the Hochschild complex $X$ of $\Pi$. By
assumption the differential $d(a)$ is zero, so $a$ is an element in
$HH_p(\Pi)$. Now we claim that $a$ is a nonzero element in
$HH_p(\Pi)$.

First, the superpotential $W$ is a linear combination of paths of
length at least 3, so $d(\widetilde{Q}_1^V) \subseteq
{\mathfrak{m}}^2$, where $\mathfrak{m}$ is the two-sided ideal of
$\Pi$ generated by the arrows of $\widetilde{Q}^V$. Second, it is
obvious that the relation $\mbox{Im} {\partial}_1 \cap \{
\mbox{loops\, of}\,\,{\widetilde{Q}}^V \} = \emptyset$ holds.
Therefore, the loop $a$ can not be written in the form $\sum
d(\gamma) + \sum
\partial_1 (u \otimes v)$ for paths $\gamma \in e_i {\mathfrak{m}}
e_i$ and $u, v$ paths of ${\widetilde{Q}}^V$, which means that $a$
is a nonzero element in $HH_p(\Pi)$.

Note that the trivial paths associated to the vertices of $Q$ are
nonzero elements in $HH_0 (\Pi)$. Hence, we get a contradiction to
the isomorphism in the assumption. As a result, the quiver
${\overline{Q}}^V$ does not contain loops with zero differential and
of degree $-p$.
\end{proof}




\begin{cor}\label{41}
Keep the notation as in Theorem \ref{35} and let $m = 1$. Suppose
that $\Pi$ is $1$-rigid. Then the new quiver $Q'$ does not contain
loops.
\end{cor}

\begin{proof}

It follows from statement 1) in Theorem \ref{35} that $E$ is
quasi-isomorphic to some good completed deformed preprojective dg
algebra $\Pi' = {\widehat{\Pi}}(Q',3,W')$. Then following Theorem
\ref{33} and the analysis before Theorem \ref{35}, we can obtain
that the dg algebras $\Pi$ and $\Pi'$ have isomorphic Hochschild
homology. Therefore, the new dg algebra $\Pi'$ is also $1$-rigid.
Note that every arrow of $Q'$ has zero degree and thus has zero
differential. Hence, by Proposition \ref{40} the quiver $Q'$ does
not contain loops.
\end{proof}

\begin{conj}\label{43}
Let $\Pi = {\widehat{\Pi}(Q,m+2,W)}$ be an $m$-rigid good completed
deformed preprojective dg algebra whose zeroth homology $H^0 \Pi$ is
finite dimensional. Then any liftable almost complete $m$-cluster
tilting object has exactly $m+1$ complements in ${\mathcal
{C}}_{\Pi}$.
\end{conj}

Following the same procedure as in the proof of Corollary \ref{41},
we know that the good completed deformed preprojective dg algebra
$\Pi' = {\widehat{\Pi}(Q',m+2,W')}$ in Theorem \ref{35} is also
$m$-rigid, and the new quiver $Q'$ does not contain loops of degree
zero. It seems that we would like to get a stronger result than
Proposition \ref{40}, that is, $m$-rigidity implies that
${\overline{Q'}}^V$ does not contain loops (not only loops with zero
differential). If this is true, then it follows from statement 2) in
Theorem \ref{35} that any liftable almost complete $m$-cluster
tilting object has exactly $m+1$ complements in ${\mathcal
{C}}_{\Pi}$.

If Conjecture \ref{43} holds, then the $m$-rigidity property shown
in Proposition \ref{42} of the dg algebra $\Pi =
{\widehat{\Pi}}(Q,m+2,0)$ with $Q$ an acyclic quiver implies that
any liftable almost complete $m$-cluster tilting object in
${\mathcal {C}}_{\Pi}$ has exactly $m+1$ complements. Later
Proposition \ref{44} shows that any almost complete $m$-cluster
tilting object in ${\mathcal {C}}_{\Pi}$ is liftable in the `acyclic
quiver' case. Thus, on one hand, if Conjecture \ref{43} holds, we
can deduce a common result both in \cite{W} and \cite{ZZ}, namely,
any almost complete $m$-cluster tilting object in the classical
$m$-cluster category ${\mathcal {C}}_Q^{(m)}$ has exactly $m+1$
complements. On the other hand, it follows from this common result
for the classical $m$-cluster category ${\mathcal {C}}_Q^{(m)}$,
which is triangle equivalent to the corresponding generalized
$m$-cluster category ${\mathcal {C}}_{\Pi}$, that any almost
complete $m$-cluster tilting object in ${\mathcal {C}}_{\Pi}$ should
have exactly $m+1$ complements.

\vspace{.3cm}

\section{A long exact sequence and the acyclic case}

Let $A$ be a dg algebra satisfying Assumptions \ref{23}. In the
first part of this section, we give a long exact sequence to see the
relations between extension spaces in generalized $m$-cluster
categories ${\mathcal {C}}_A$ and extension spaces in derived
categories ${\mathcal {D}} (= {\mathcal {D}}(A))$. If the extension
spaces between two objects of ${\mathcal {C}}_A$ are zero, in some
cases, we can deduce that the extension spaces between these two
objects are also zero in the derived category $\mathcal {D}$.

\begin{prop} \label{14}
Suppose that $X$ and $Y$ are two objects in the fundamental domain
$\mathcal {F}$. Then there is a long exact sequence $$0 \ra {\rm
Ext}^1_{\mathcal {D}}(X,Y) \ra {\rm Ext}^1_{{\mathcal {C}}_A}(X,Y)
\ra D{\rm Ext}^m_{\mathcal {D}}(Y,X)$$ $$ \ra {\rm Ext}^2_{\mathcal
{D}}(X,Y) \ra {\rm Ext}^2_{{\mathcal {C}}_A}(X,Y) \ra D{\rm
Ext}^{m-1}_{\mathcal {D}}(Y,X)
$$ $$\ra \quad \cdots \quad\quad \cdots \quad \ra$$ $${\rm
Ext}^m_{\mathcal {D}}(X,Y) \ra {\rm Ext}^m_{{\mathcal {C}}_A}(X,Y)
\ra D{\rm Ext}^1_{\mathcal {D}}(Y,X) \ra 0 .$$
\end{prop}

\begin{proof}
We have the canonical triangle
$${\tau}_{\leq -m}X \ra X \ra {\tau}_{\geq 1-m}X \ra \Sigma ({\tau}_{\leq
-m}X) ,$$ which yields the long exact sequence $$\cdots \ra {\rm
Hom}_{\mathcal {D}}({\Sigma}^{-t} ({\tau}_{\geq 1-m}X), Y) \ra {\rm
Hom}_{\mathcal {D}}({\Sigma}^{-t} X, Y) \ra {\rm Hom}_{\mathcal
{D}}({\Sigma}^{-t} ({\tau}_{\leq -m}X), Y) \ra \cdots , \quad t \in
\mathbb{Z}.$$

\smallskip

Step 1. {\it The isomorphism $${\rm Hom}_{\mathcal
{D}}({\Sigma}^{-t}({\tau}_{\geq 1-m}X),Y) \simeq D{\rm
Hom}_{\mathcal {D}}(Y,{\Sigma}^{m+2-t}X)$$ holds when $t \leq m+1$.}

By the Calabi-Yau property, there holds the isomorphism
$${\rm Hom}_{\mathcal {D}}({\Sigma}^{-t}({\tau}_{\geq 1-m}X),Y) \simeq
D{\rm Hom}_{\mathcal {D}}(Y,{\Sigma}^{m+2-t}({\tau}_{\geq 1-m}X)),
\quad t \in \mathbb{Z}. \quad \quad {\rm (8.1)}.$$ Applying the
functor ${\rm Hom}_{\mathcal {D}}(Y,-)$ to the triangle which we
start with, we obtain the exact sequence
$$(Y,{\Sigma}^{m+2-t}({\tau}_{\leq -m}X)) \ra
(Y,{\Sigma}^{m+2-t}X)  \ra (Y,{\Sigma}^{m+2-t}({\tau}_{\geq 1-m}X))
\ra (Y,{\Sigma}^{m+3-t}({\tau}_{\leq -m}X)),$$where $(-,-)$ denotes
${\rm Hom}_{\mathcal {D}}(-,-)$. When $t \leq m+1$, we have that
$(-m)-(m+2-t) \leq -m-1$. Then the objects
${\Sigma}^{m+2-t}({\tau}_{\leq -m}X)$ and
${\Sigma}^{m+3-t}({\tau}_{\leq -m}X)$ belong to ${\mathcal
{D}}^{\leq -m-1}$. Note that $Y$ is in
$^{\perp}{\mathcal{D}}^{{\leq} {-m-1}}$. Therefore, the following
isomorphism holds
$$\quad \quad \quad \quad {\rm Hom}_{\mathcal {D}}(Y,{\Sigma}^{m+2-t}({\tau}_{\geq
1-m}X)) \simeq {\rm Hom}_{\mathcal {D}}(Y,{\Sigma}^{m+2-t}X). \quad
\quad \quad \quad \quad  {\rm (8.2).}$$ As a consequence, when $t
\leq m+1$, together by (8.1) and (8.2), we have the isomorphism
$${\rm Hom}_{\mathcal {D}}({\Sigma}^{-t}({\tau}_{\geq 1-m}X),Y)
\simeq D{\rm Hom}_{\mathcal {D}}(Y,{\Sigma}^{m+2-t}X).$$ Moreover,
if $t \leq 1$, the object ${\Sigma}^{m+2-t}X$ belongs to ${\mathcal
{D}}^{\leq -m-1}$, so the space ${\rm Hom}_{\mathcal
{D}}(Y,{\Sigma}^{m+2-t}X)$ vanishes, and so does the space ${\rm
Hom}_{\mathcal {D}}({\Sigma}^{-t}({\tau}_{\geq 1-m}X),Y)$.

\smallskip

Step 2. {\it When $t \leq m$, we have the following isomorphism
$${\rm Hom}_{\mathcal {D}}({\Sigma}^{-t}({\tau}_{\leq -m}X),Y)
\simeq {\rm Hom}_{{\mathcal {C}}_A}(\pi X,{\Sigma}^t(\pi Y)).$$}

Consider the triangles $${\tau}_{\leq s-1}X \ra {\tau}_{\leq s}X \ra
{\Sigma}^{-s}(H^sX) \ra {\Sigma}({\tau}_{\leq s-1}X) , \quad \quad s
\in \mathbb{Z}.$$ Applying the functor ${\rm Hom}_{\mathcal
{D}}(-,Y)$ to these triangles, we can obtain the following long
exact sequences $$\cdots \ra ({\Sigma}^{-s-t}(H^sX),Y) \ra
({\Sigma}^{-t}({\tau}_{\leq s}X),Y) \ra ({\Sigma}^{-t}({\tau}_{\leq
s-1}X),Y) \ra ({\Sigma}^{-s-t-1}(H^sX),Y) \ra \cdots,$$where $(-,-)$
denotes ${\rm Hom}_{\mathcal {D}}(-,-)$. Using the Calabi-Yau
property, we have that $${\rm Hom}_{\mathcal
{D}}({\Sigma}^{-s-t}(H^sX),Y) \simeq D{\rm Hom}_{\mathcal
{D}}(Y,{\Sigma}^{m+2-s-t}(H^sX)) , \quad \quad t \in \mathbb{Z}.$$
When $t \leq -s$, the inequality $m+2-s-t-1 \geq m+1$ holds. So the
two objects ${\Sigma}^{m+2-s-t}(H^sX)$ and
${\Sigma}^{m+2-s-t-1}(H^sX)$ belong to ${\mathcal {D}}^{\leq -m-1}$.
Therefore, ${\rm Hom}_{\mathcal {D}}({\Sigma}^{-s-t}(H^sX),Y)$ and
the space ${\rm Hom}_{\mathcal {D}}({\Sigma}^{-s-t-1}(H^sX),Y)$ are
zero, and the following isomorphism
$${\rm Hom}_{\mathcal {D}}({\Sigma}^{-t}({\tau}_{\leq s}X),Y)
\simeq {\rm Hom}_{\mathcal {D}}({\Sigma}^{-t}({\tau}_{\leq
s-1}X),Y)$$ holds. As a consequence, we can get the following
isomorphisms $${\rm Hom}_{\mathcal {D}}({\Sigma}^{-t}({\tau}_{\leq
-t}X),Y) \simeq {\rm Hom}_{\mathcal {D}}({\Sigma}^{-t}({\tau}_{\leq
-t-1}X),Y) \simeq \cdots \simeq {\rm Hom}_{\mathcal
{D}}({\Sigma}^{-t}({\tau}_{\leq -m}X),Y),\, t \leq m. \, {\rm
(8.3)}.$$ Since the functor $\pi : {\rm per}A \ra {\mathcal {C}}_A$
induces an equivalence from ${\Sigma}^{t}{\mathcal {F}}$ to
$\mathcal {C}$ (Proposition \ref{3} applies to shifted
$t$-structure), the following bijections are true
$${\rm Hom}_{{\mathcal {C}}_A}(\pi X,{\Sigma}^t(\pi Y)) \simeq {\rm Hom}_{{\mathcal
{C}}_A}(\pi ({\tau}_{\leq -t}X),\pi ({\Sigma}^{t}Y)) \simeq {\rm
Hom}_{\mathcal {D}}({\tau}_{\leq -t}X, {\Sigma}^t Y).\, \quad \quad
\quad {\rm (8.4)}.$$ Hence, when $t \leq m$, together by (8.3) and
(8.4), we have the isomorphism
$${\rm Hom}_{\mathcal {D}}({\Sigma}^{-t}({\tau}_{\leq -m}X),Y)
\simeq {\rm Hom}_{{\mathcal {C}}_A}(\pi X,{\Sigma}^t(\pi Y)).$$

Therefore, the long exact sequence at the beginning becomes $$0 =
{\rm Hom}_{\mathcal {D}}({\Sigma}^{-1}({\tau}_{\geq 1-m}X),Y) \ra
{\rm Ext}^1_{\mathcal {D}}(X,Y) \ra {\rm Ext}^1_{{\mathcal
{C}}_A}(X,Y) \ra D{\rm Ext}^m_{\mathcal {D}}(Y,X)$$ $$ \ra {\rm
Ext}^2_{\mathcal {D}}(X,Y) \ra {\rm Ext}^2_{{\mathcal {C}}_A}(X,Y)
\ra D{\rm Ext}^{m-1}_{\mathcal {D}}(Y,X)
$$ $$\ra \quad \cdots \quad\quad \cdots \quad \ra$$ $${\rm
Ext}^m_{\mathcal {D}}(X,Y) \ra {\rm Ext}^m_{{\mathcal {C}}_A}(X,Y)
\ra D{\rm Ext}^1_{\mathcal {D}}(Y,X) \ra {\rm Hom}_{\mathcal
{D}}({\Sigma}^{-m-1}X,Y) = 0 .$$ This concludes the proof.
\end{proof}

\begin{rems}

1) When $m =1$, the long exact sequence in Proposition \ref{14}
becomes the following short exact sequence (already appearing in the
proof of Proposition \ref{39})
$$\quad \quad 0 \ra {\rm Ext}^1_{\mathcal {D}}(X,Y) \ra {\rm
Ext}^1_{{\mathcal {C}}_A}(X,Y) \ra D{\rm Ext}^1_{\mathcal {D}}(Y,X)
\ra 0  \quad \quad {\rm (8.5)},$$ which was presented in \cite{Am08}
for the Hom-finite $2$-Calabi-Yau case, and also was presented in
\cite{Pla} for the Jacobi-infinite $2$-Calabi-Yau case.

2) If $T$ is an object in the fundamental domain $\mathcal {F}$
satisfying $${\rm Ext}^{i}_{\mathcal {D}}(T,T) = 0, \quad i=1,
\ldots, m,$$ then the long exact sequence in Proposition \ref{14}
implies that the spaces ${\rm Ext}^i_{{\mathcal {C}}_A}(T,T)$ also
vanish for integers $1 \leq i \leq m$.
\end{rems}


Suppose that $X$ and $Y$ are two objects in the fundamental domain.
It is clear that ${\rm Ext}^i_{\mathcal {D}}(X,Y)$ vanishes when $i
> m$, since X belongs to $\mathcal {F}$ and ${\Sigma}^i Y$ lies in ${\mathcal {D}}^{\leq
-m-1}$. Now we assume that the spaces ${\rm Ext}^i_{{\mathcal
{C}}_A}(X,Y)$ are zero for integers $1 \leq i \leq m$. What about
the extension spaces ${\rm Ext}^i_{\mathcal {D}}(X,Y)$ in the
derived category? Do they always vanish?

When $m=1$, the short exact sequence (8.5) implies that the space
${\rm Ext}^1_{\mathcal {D}}(X,Y)$ vanishes.

When $m > 1$, we will give the answer for completed Ginzburg dg
categories (the same as completed deformed preprojective dg algebras
in this case) arising from acyclic quivers.

\begin{prop}\label{21}
Let $Q$ be an acyclic quiver. Let $\Gamma$ be the completed Ginzburg
dg category ${\widehat{\Gamma}_{m+2}}(Q,0)$ and ${\mathcal
{C}}_{\Gamma}$ the generalized $m$-cluster category. Suppose that
$X$ and $Y$ are two objects in the fundamental domain $\mathcal {F}$
which satisfy
$${\rm Ext}^i_{{\mathcal {C}}_{\Gamma}}(X,Y) = 0, \quad i = 1, \ldots, m.$$
Then the extension spaces ${\rm Ext}^i_{{\mathcal
{D}}(\Gamma)}(X,Y)$ vanish for all positive integers $i$.
\end{prop}

\begin{proof}
Let $B$ be the path algebra $kQ$ and $\Omega$ the inverse dualizing
complex $R {\rm Hom}_{B^e} (B, B^e)$. Set $\Theta =
{\Sigma}^{m+1}\Omega$. Then the $(m+2)$-Calabi-Yau completion
\cite{Ke09} of $B$ is the tensor dg category $${\Pi}_{m+2}(B) =
T_B(\Theta) = B \oplus \Theta \oplus (\Theta \otimes_B \Theta)
\oplus \ldots .$$ Theorem 6.3 in \cite{Ke09} shows that
${\Pi}_{m+2}(B)$ is quasi-isomorphic to the completed Ginzburg dg
category $\Gamma$. Thus, we can write $\Gamma$ as
$$\Gamma = B \oplus \Theta \oplus (\Theta \overset{L}\otimes_B
\Theta) \oplus \ldots = {\oplus}_{p \geq
0}{\Theta}^{\overset{L}\otimes_B p}.$$ Let $X', \, Y'$ be two
objects in ${\mathcal {D}}_{fd}(B)$. The following isomorphisms hold
$${\rm Hom}_{\mathcal {D}(\Gamma)}(X'\overset{L}\otimes_B\Gamma,
Y'\overset{L}\otimes_B\Gamma) \simeq {\rm Hom}_{{\mathcal
{D}}(B)}(X', Y'\overset{L}\otimes_B\Gamma|_B) \simeq {\rm
Hom}_{{\mathcal {D}}(B)}(X', Y' \overset{L}\otimes_B({\oplus}_{p
\geq 0}{\Theta}^{\overset{L}\otimes_B p}))$$ $$ \simeq {\rm
Hom}_{{\mathcal {D}}(B)}(X', \oplus_{p \geq 0}
(Y'\overset{L}\otimes_B{\Theta}^{\overset{L}\otimes_B p})) \simeq
\oplus_{p \geq 0}{\rm Hom}_{{\mathcal {D}}(B)}(X',
Y'\overset{L}\otimes_B{\Theta}^{\overset{L}\otimes_B p}).$$ By Lemma
\ref{24}, the category ${\mathcal {D}}_{fd}(B)$ admits a Serre
functor $S$ whose inverse is $- \overset{L}\otimes_B \Omega.$
Therefore, the functor $- \overset{L}\otimes_B \Theta$ is equal to
the functor $S^{-1}{\Sigma}^{m+1} (\simeq {\tau}^{-1}{\Sigma}^m)$,
where $\tau$ is the Auslander-Reiten translation. As a consequence,
we have that
$${\rm Hom}_{\mathcal {D}(\Gamma)}(X'\overset{L}\otimes_B\Gamma,
Y'\overset{L}\otimes_B\Gamma) \simeq \oplus_{p \geq 0}{\rm
Hom}_{{\mathcal {D}}_{fd}(B)}(X', (\tau^{-1}{\Sigma}^m)^pY').$$

Let ${\mathcal {C}}_{Q}^{(m)}$ be the $m$-cluster category
${\mathcal {D}}_{fd}(B)/({\tau^{-1}\Sigma^m})^{\mathbb{Z}}$.
Consider the following commutative diagram
\[
\xymatrix@C=2.5CM{{\mathcal {D}}_{fd}(B) \ar[d]_{\pi_B} \ar[r]^{-
\overset{L}\otimes_B \Gamma} &
{\rm per}\Gamma \ar[d]^{\pi_\Gamma}  \\
{\mathcal {C}}_{Q}^{(m)} \ar@<1ex>[r]_{\simeq}^{ -
\overset{L}\otimes_B \Gamma } & {\mathcal {C}}_{\Gamma} . }
\]
Under the equivalence, let $X= X' \overset{L}\otimes_B \Gamma$ and
$Y= Y' \overset{L}\otimes_B \Gamma$, so the vanishing of spaces
${\rm Ext}^i_{{\mathcal {C}}_{\Gamma}}(X,Y)$ implies that ${\rm
Ext}^i_{{\mathcal {C}}_Q^{(m)}}(X',Y')$ also vanish for integers $1
\leq i \leq m$. Note that $${\rm Ext}^i_{{\mathcal
{C}}_Q^{(m)}}(X',Y') \simeq \oplus_{p \in \mathbb{Z}} {\rm
Ext}^i_{{\mathcal {D}}_{fd}(B)}(X',(\tau^{-1}\Sigma^m)^pY').$$
Hence, we obtain that$${\rm Ext}^i_{{\mathcal {D}}(\Gamma)}(X,Y)
\simeq \oplus_{p \geq 0} {\rm Ext}^i_{{\mathcal
{D}}_{fd}(B)}(X',(\tau^{-1}\Sigma^m)^pY') = 0, \quad \quad 1 \leq i
{\leq} m.$$
\end{proof}

\smallskip

Let $Q$ be an ordinary acyclic quiver and $B$ the path algebra $kQ$.
Let $\Gamma$ be its completed Ginzburg dg category
$\widehat{\Gamma}_{m+2}(Q,0)$. Let $T$ be an $m$-cluster tilting
object in ${\mathcal {C}}_Q^{(m)}$. Then $T$ is induced from an
object $T'$ (that is, $T = \pi(T')$) in the fundamental domain
$${\mathcal {S}}_m := {\mathcal {S}}^0_m \vee {\Sigma}^m B, \quad {\rm where}
\,\, {\mathcal {S}}^0_m := {\rm mod}B \vee {\Sigma}({\rm mod}B)
\ldots \vee {\Sigma}^{m-1}({\rm mod}B).$$

\begin{lem}[\cite{BRT}] \label{37}
The object $T'$ is a partial silting object, that is,
$${\rm Hom}_{{\mathcal {D}}_{fd}(B)}(T', {\Sigma}^i T') = 0, \quad \quad i > 0;$$
and $T'$ is maximal with this property.
\end{lem}

An object in ${\mathcal {D}}_{fd}(B)$ which satisfies the `maximal
partial silting' property as in Lemma \ref{37} is called a `silting'
object in \cite{BRT}. Next we will show that our definition for
silting object in per$B$ coincides with their definition.

\begin{lem}\label{45}
Let $U$ be a basic partial silting object in ${\mathcal
{D}}_{fd}(B)$. Then $U$ is maximal partial silting if and only if
$U$ generates ${\rm per}B$.
\end{lem}

\begin{proof}
On one hand, assume that $U$ is a basic partial silting object and
generates per$B$. By Lemma \ref{38}, the object $U$ has the same
number of indecomposable direct summands as that of the dg algebra
$B$ itself. That is, $U$ is a basic partial silting object with
$|Q_0|$ indecomposable direct summands. Following from Lemma 2.2 in
\cite{BRT}, we obtain that $U$ is a maximal partial silting object.

On the other hand, assume that $U$ is a maximal partial silting
object in ${\mathcal {D}}_{fd}(B)$. We decompose $U$ into a direct
sum ${\Sigma}^{k_1}U_1 \oplus \ldots \oplus {\Sigma}^{k_r}U_r$ such
that each $U_i$ lies in ${\rm mod}B$ and $k_1 < \ldots < k_r$. Set
$U' = {\oplus}^{r}_{i=1} U_i$. It follows from Lemma 2.2 in
\cite{BRT} that the object $U'$ can be ordered to a complete
exceptional sequence. Let $C(U')$ be the smallest full subcategory
of ${\rm mod}B$ which contains $U'$ and is closed under extensions,
kernels of epimorphisms, and cokernels of monomorphisms. By Lemma 3
in \cite{CB}, the subcategory $C(U')$ is equal to ${\rm mod}B$. As a
consequence, the object $U$ generates ${\mathcal {D}}_{fd}(B)$ which
is equal to per$B$.
\end{proof}

Since $B$ is finite dimensional and hereditary, the subcategory
${\mathcal {S}}^0_m$ is contained in $^{\perp}{\mathcal
{D}}(B)^{\leq -m-1}$. The isomorphism $${\rm Hom}_{{\mathcal
{D}}(B)} ({\Sigma}^{m} B, M) \simeq H^mM \quad (M \in {\mathcal
{D}}(B))$$ implies that ${\Sigma}^m B$ is in $^{\perp}{\mathcal
{D}}(B)^{\leq -m-1}$. So ${\mathcal {S}}_m$ is contained in
${\mathcal {D}}(B)^{\leq 0} \cap ^{\perp}{\mathcal {D}}(B)^{\leq
-m-1} \cap {\mathcal {D}}_{fd}(B)$.

Set $Z = T' \overset{L}\otimes_{B}\Gamma$. For any object $N$ in
${\mathcal {D}}(\Gamma)$, we have the following canonical
isomorphism $${\rm Hom}_{{\mathcal {D}}(\Gamma)}(T' \overset{L}
\otimes_B \Gamma, N) \simeq {\rm Hom}_{{\mathcal {D}}(B)}(T', \,
{\rm RHom}_{\Gamma}(\Gamma,N)).$$ When $N$ lies in ${\mathcal
{D}}(\Gamma)^{\leq -m-1}$, the right hand side of the above
isomorphism becomes zero. Thus, the object $Z$ is in the fundamental
domain of ${\mathcal {D}}(\Gamma)$. The spaces ${\rm
Ext}^i_{{\mathcal {C}}_Q^{(m)}}(T,T)$ vanish for integers $1 \leq i
\leq m$, following the proof of Proposition \ref{21}, the space
${\rm Hom}_{{\mathcal {D}}(\Gamma)}(Z,\Sigma^iZ)$ is zero for each
positive integer $i$. In addition, Lemma \ref{37} and Lemma \ref{45}
together imply that $T'$ generates ${\mathcal {D}}_{fd}(B)$. Hence,
the object $Z$ generates per$\Gamma$. So $Z$ is a basic silting
object whose image in ${\mathcal {C}}_{\Gamma}$ is $T \overset{L}
\otimes_B \Gamma$.

Now we conclude the above analysis to get the following proposition.

\begin{prop}\label{44}
Let $Q$ be an acyclic quiver and $B$ its path algebra. Let $\Gamma$
be the completed Ginzburg dg category
${\widehat{\Gamma}_{m+2}}(Q,0)$ and ${\mathcal {C}}_{\Gamma}$ the
generalized $m$-cluster category. Then any $m$-cluster tilting
object in ${\mathcal {C}}_{\Gamma}$ is induced by a silting object
in $\mathcal {F}$ under the canonical projection $\pi: {\rm
per}\Gamma \rightarrow {\mathcal {C}}_{\Gamma}$.
\end{prop}

\begin{proof}
Let $\overline{T}$ be an $m$-cluster tilting object in ${\mathcal
{C}}_{\Gamma}$. Then $\overline{T}$ can be written as $T \overset{L}
\otimes_B \Gamma$ for some $m$-cluster tilting object $T$ in
${\mathcal {C}}_Q^{(m)}$, where $T$ is induced by some silting
object $T'$ in ${\mathcal {D}}_{fd}(B)$. The object $T'
\overset{L}\otimes_{B}\Gamma$ (denoted by $Z$) is a silting object
in the fundamental domain ${\mathcal {F}} (\subseteq \mbox{per}
{\Gamma})$ whose image under the canonical projection $\pi: {\rm
per}\Gamma \rightarrow {\mathcal {C}}_{\Gamma}$ is equal to
$\overline{T}$. This completes the proof.
\end{proof}



\end{document}